\documentclass[10pt,reqno]{article}
\usepackage{amssymb, amsmath, amsthm, amsfonts, amscd, epsfig}
\usepackage[english]{babel}
\usepackage{bm}
\usepackage{graphicx, epsfig}
\newtheorem{theorem}{Theorem}[section]

\newtheorem{lemma}[theorem]{Lemma}
\newtheorem{prop}[theorem]{Proposition}
\newtheorem{definition}{Definition}

\newcommand{\nm}{\noalign{\smallskip}}

\def\ep{\varepsilon}

\newcommand{\bE}{\mathbf{E}}
\newcommand{\Bp}{\mathbf{p}}
\newcommand{\Bq}{\mathbf{q}}
\newcommand{\Bc}{\mathbf{c}}
\newcommand{\bH}{\mathbf{H}}

\newcommand{\Bx}{\mathbf{x}}
\newcommand{\By}{\mathbf{y}}

\newcommand{\RR}{\mathbb{R}}

\newcommand{\Scal}{\mathcal{S}}
\newcommand{\Ecal}{\mathcal{E}}
\newcommand{\Hcal}{\mathcal{H}}
\newcommand{\p}{\partial}
\newcommand{\Ga}{\alpha}

\newcommand{\GG}{\Gamma}

\newcommand{\pd}[2]{\frac {\p #1}{\p #2}}
\newcommand{\ds}{\displaystyle}
\newcommand{\eqnref}[1]{(\ref {#1})}

\renewcommand{\qed}{\hfill $\Box$ \medskip}
\newcommand{\beq}{\begin{equation}}
\newcommand{\eeq}{\end{equation}}

\numberwithin{equation}{section}
\numberwithin{figure}{section}

\begin{document}

\title{Enhancement of near cloaking for the full Maxwell equations\thanks{\footnotesize
This work was supported  by ERC Advanced Grant Project
MULTIMOD--267184, TJ Park Junior Faculty Fellowship, and National Research Foundation through grants
No. 2010-0017532, 2010-0004091, and 2012-003224.}}

\author{Habib Ammari\thanks{\footnotesize Department of Mathematics and Applications, Ecole Normale Sup\'erieure,
45 Rue d'Ulm, 75005 Paris, France (habib.ammari@ens.fr).} \and
Hyeonbae Kang\thanks{Department of Mathematics, Inha University,
Incheon 402-751, Korea (hbkang@inha.ac.kr, hdlee@inha.ac.kr).}
\and Hyundae Lee\footnotemark[3] \and Mikyoung
Lim\thanks{\footnotesize Department of Mathematical Sciences,
Korean Advanced Institute of Science and Technology, Daejeon
305-701, Korea (mklim@kaist.ac.kr, shyu@kaist.ac.kr).} \and
Sanghyeon Yu\footnotemark[4]}

\maketitle

\begin{abstract}

In this paper, we consider near cloaking for the full Maxwell
equations. We extend the method of \cite{AKLL1,AKLL2}, where the
quasi-static limit case and the Helmholtz equation are considered,
to electromagnetic scattering problems. We construct very
effective near cloaking structures for the electromagnetic
scattering problem at a fixed frequency. These new structures are,
before using the transformation optics, layered structures and are
designed so that their first scattering coefficients vanish.
Inside the cloaking region, any target has near-zero scattering
cross section for a band of frequencies. We analytically show that
our new construction significantly enhances the cloaking effect
for the full Maxwell equations.
\end{abstract}



\section{Introduction}

The cloaking problem is to make a target invisible from far-field
electromagnetic measurements \cite{pendry, leonhardt, miller,
GKLU, BL,broaband}. Many schemes for cloaking are under active
current investigation. These include exterior cloaking in which
the cloaking region is outside the cloaking device
\cite{MNMP_PRSA_05, alu, MN_PRSA_06, bruno,bouchitte, inv}, active
cloaking \cite{GMO_09}, and interior cloaking, which is the focus
of our study.

In interior cloaking, the difficulty is to construct
electromagnetic material parameter distributions of a cloaking
structure such that any target placed inside the structure is
undetectable to  waves. One approach is to use transformation
optics \cite{pendry, GKLU, mit2, crossect, mit1}. It takes
advantage of the fact that the equations governing
electromagnetism have transformation laws under change of
variables. This allows one to design structures that steel waves
around a hidden region, returning them to their original path on
the far side. The change of variables based cloaking method uses a
singular transformation to boost the material properties so that
it makes a cloaking region look like a point to outside
measurements. However, this transformation induces the singularity
of material constants in the transversal direction (also in the
tangential direction in two dimensions), which causes difficulty
both in the theory and applications. To overcome this weakness, so
called `near cloaking' is naturally considered, which is a
regularization or an approximation of singular cloaking. In
\cite{kohn1}, instead of the singular transformation, the authors
use a regular one to push forward the material constant in the
conductivity equation describing the quasi-static limit of
electromagnetism, in which a small ball is blown up to the
cloaking region. In \cite{kohn2}, this regularization point of
view is adopted for the Helmholtz equation. See also \cite{liu,
nguyen}. More recently, Bao and Liu \cite{gang} considered near
cloaking for the full Maxwell equations. They derived sharp
estimates for the boundary effect due to a small inclusion with an
arbitrary material parameters enclosed by a thin high-conducting
layer. Their results show that the near cloaking scheme can be
applied to cloak targets from electromagnetic boundary
measurements.

In \cite{AKLL1,AKLL2} it is shown that the near cloaking, from
measurements of the Dirichlet-to-Neumann map for the conductivity
equation and of the scattering cross section for the Helmholtz
equation, can be drastically enhanced by using multi-layered
structures. The structures are designed so that their generalized
polarization tensors (GPTs) or scattering coefficients vanish (up
to a certain order). GPTs are building blocks of the far-field
behavior of solutions in the quasi-static limits (conductivity
equations) and the scattering coefficients are `Fourier
coefficients' of the scattering amplitude. The multi-layered
structures combined with the usual change of variables
(transformation optics) greatly reduce the visibility of an
object. This fact is also confirmed by numerical experiments
\cite{AKLL3}.

The purpose of this paper is to extend the results of
\cite{AKLL1,AKLL2} to Maxwell's equation and show that the near
cloaking from cross section scattering measurements at a fixed
frequency can be enhanced by using layered structures together
with the change of variables. Again the layered structures are
designed so that their first scattering coefficients vanish. It is
also shown that inside the cloaking region, any target has
near-zero scattering cross section for a band of (low)
frequencies. We analytically show that our new construction
significantly enhances the near cloaking effect for the full
Maxwell equations. It is worth mentioning that even if the basic
scheme of this work is parallel to that of \cite{AKLL2}, the
analysis is much more complicated due to the vectorial nature of
the Maxwell equations.

The paper is organized as follows. In Section 2, we recall some
fundamental results on the scattering problem for the full Maxwell
equations. In Section 3, we introduce the scattering coefficients
of an electromagnetic inclusion and prove that the scattering
coefficients are basically the spherical harmonic expansion
coefficients of the far-field pattern. Section 4 is devoted to the
construction of layered structures with vanishing scattering
coefficients. We also present some numerical examples of the scattering
coefficient vanishing structures. In Section 5, we show that the near cloaking is
enhanced if a scattering coefficient vanishing structure is used.

\section{Multipole solutions to the Maxwell equations}

 In this section, we recall a few fundamental results related to
 electromagnetic scattering, which will be essential
in the sequel.

Consider the time-dependent Maxwell equations
\begin{equation*}
 \ \left \{
 \begin{array}{l}
\nabla\times\Ecal = -\mu\pd{}{t}\Hcal,\\
\nabla\times\Hcal = \epsilon\pd{}{t}\Ecal,
 \end{array}
 \right .
 \end{equation*}
where $\mu$ is the magnetic permeability and $\epsilon$ is the
electric permittivity.

In  the time-harmonic regime, we look for the electromagnetic
fields of the form
\begin{equation*}
 \ \left \{
 \begin{array}{l}
\ds\Hcal(\Bx,t)={\bH}(\Bx)e^{-i\omega t},\\
\ds\Ecal(\Bx,t) = {\bE}(\Bx)e^{-i\omega t},
 \end{array}
 \right.
 \end{equation*}
where $\omega$ is the frequency. The couple $(\bE,\bH)$ is a
solution to the harmonic Maxwell equations
\beq\label{Meqn}
 \ \left \{
 \begin{array}{l}
\ds\nabla\times{\bE} = i\omega\mu{\bH},\\
\ds\nabla\times{\bH} = -i\omega\epsilon{\bE}.
 \end{array}
 \right .
 \eeq
We say that $(\bE,\bH)$ is radiating if it satisfies the
Silver-M\"{u}ller radiation condition:
 $$\lim_{|\Bx|\rightarrow\infty} |\Bx| (\sqrt{\mu}\bH\times\hat{\Bx}-\sqrt{\epsilon}\bE)=0,$$
 where $\hat{\Bx} = \Bx/|\Bx|$.
In the sequel, we set $k=\omega\sqrt{\epsilon\mu}$, which is called the wave number.

For $m=-n,\dots,n$ and $n=1,2,\dots$, set $Y_n^m$ to be the spherical harmonics defined on the
 unit sphere $S$. For a wave number $k>0$, the following function
 \begin{equation}\label{vnm}
 \ds v_{n,m}(k;\Bx) = h_n^{(1)}(k|\Bx|)Y_n^m(\hat{\Bx})
 \end{equation}
 satisfies the Helmholtz equation $\Delta v+k^2 v=0$ in $\RR^3\setminus\{0\}$ and the Sommerfeld
 radiation condition:
$$
\lim_{|\Bx|\rightarrow\infty} |\Bx| (\frac{\partial
v_{n,m}}{\partial |\Bx|} (k;\Bx) - i k v_{n,m}(k;\Bx) ) = 0.
$$
 Here, $h_n^{(1)}$ is the spherical Hankel function of the first kind and order $n$ which satisfies the Sommerfeld
 radiation condition. Similarly,
 $\tilde{v}_{n,m}(\Bx)$ is defined as
 \begin{equation}\label{vtnm}
\ds \tilde{v}_{n,m}(k;\Bx)=j_n(k|\Bx|)Y_n^m(\hat{\Bx}),
  \end{equation}
where $j_n$ is the spherical Bessel function of the first kind.
The function $\tilde{v}_{n,m}$ satisfies the Helmholtz equation in all $\RR^3$.

In the same manner, we can make solutions to the Maxwell system
with the vector version of spherical harmonics. Define the vector
spherical harmonics as
 \begin{equation}\label{vsh}
 \ds\mathbf{U}_{n,m} = \frac{1}{\sqrt{n(n+1)}}\nabla_S Y_{n}^m(\hat{\Bx})
 \ds \quad\mbox{and}\quad \mathbf{V}_{n,m} = \hat{{\Bx}}\times \mathbf{U}_{n,m},
 \end{equation}
 for $m=-n,\dots,n$ and $n=1,2,\dots$.
 Here, $\hat{\Bx}\in S$ and $\nabla_{S}$ denotes the surface gradient on the unit sphere $S$. The vector spherical
 harmonics defined in \eqnref{vsh} form a complete orthogonal basis for $L^2_T(S)$,
 where $L^2_T(S)=\{\mathbf{u}\in(L^2(S))^3\
 |\ \bm{\nu} \cdot\mathbf{u}=0\}$ and $\bm{\nu}$ is the outward unit normal to $S$.

 Multiplying the vector spherical harmonics to the Hankel function, we can make the so-called
 multipole solutions to the Maxwell system. To make the analysis simple, we separate the
 solutions into transverse electric, $(\bE\cdot \Bx)=0$, and transverse magnetic, $(\bH\cdot \Bx)=0$.
 Define the exterior transverse electric multipoles to \eqnref{Meqn} as
\beq\label{multipoles}
 \ \left \{
 \begin{array}{l}
\ds \bE^{TE}_{n,m}(k;\Bx)= -\sqrt{n(n+1)}h_n^{(1)}(k|\Bx|)\mathbf{V}_{n,m}(\hat{\Bx}),\\
\nm
 \ds \bH_{n,m}^{TE}(k;\Bx) =
 -\frac{i}{\omega\mu}\nabla\times\Bigr(-\sqrt{n(n+1)} h_n^{(1)}(k|\Bx|)\mathbf{V}_{n,m}(\hat{\Bx})\Bigr),
 \end{array}
 \right .
 \eeq
 and the exterior transverse magnetic multipoles as
 \beq
 \ \left \{
 \begin{array}{l}
\ds \bE_{n,m}^{TM}(k;\Bx)= \frac{i}{\omega\epsilon}\nabla\times\Bigr(-\sqrt{n(n+1)}h_n^{(1)}(k|\Bx|)\mathbf{V}_{n,m}(\hat{\Bx})\Bigr),\\
\nm
 \ds \bH_{n,m}^{TM}(k;\Bx) =-\sqrt{n(n+1)}
h_n^{(1)}(k|\Bx|)\mathbf{V}_{n,m}(\hat{\Bx}). \end{array}
 \right .
 \eeq
The exterior electric and magnetic multipole satisfies the radiation condition.  By the same way, we define the
interior multipoles $(\widetilde{\bE}_{n,m}^{TE}, \widetilde{\bH}_{n,m}^{TE})$ and $(\widetilde{\bE}_{n,m}^{TM},
\widetilde{\bH}_{n,m}^{TM})$ with $h_n^{(1)}$ replaced by $j_n$, {\it i.e.},
\beq\label{multipoles_int}
 \ \left \{
 \begin{array}{l}
\ds \widetilde{\bE}^{TE}_{n,m}(k;\Bx)= -\sqrt{n(n+1)}j_n^{(1)}(k|\Bx|)\mathbf{V}_{n,m}(\hat{\Bx}),\\
\nm
 \ds \widetilde{\bH}_{n,m}^{TE}(k;\Bx) = -\frac{i}{\omega\mu}\nabla\times\widetilde{\bE}^{TE}_{n,m}(k;\Bx),
 \end{array}
 \right .
 \eeq
and
\beq
 \ \left \{
 \begin{array}{l}
\ds \widetilde{\bH}_{n,m}^{TM}(k;\Bx) =-\sqrt{n(n+1)}j_n^{(1)}(k|\Bx|)\mathbf{V}_{n,m}(\hat{\Bx}),\\
\nm
 \ds \widetilde{\bE}_{n,m}^{TM}(k;\Bx)= \frac{i}{\omega\epsilon}\nabla\times\widetilde{\bH}_{n,m}^{TM}(k;\Bx).
 \end{array}
 \right .
 \eeq
We will sometimes omit the wave number $k$ in the notation of the multipoles.

Note that we have
\begin{align}\label{curlE}
\nabla\times\bE^{TE}_{n,m}(k;\Bx)&=\frac{\sqrt{n(n+1)}}{|\Bx|}\mathcal{H}_n(k|\Bx|)\mathbf{U}_{n,m}(\hat{\Bx})
+\frac{n(n+1)}{|\Bx|}h_n^{(1)}(k_0|\Bx|)Y_{n}^m(\hat{\Bx})\hat{\Bx},\\\label{curlE2}
\nabla\times\widetilde{\bE}_{n,m}^{TE}(k;\Bx)&=\frac{\sqrt{n(n+1)}}{|\Bx|}\mathcal{J}_n(k|\Bx|)\mathbf{U}_{n,m}(\hat{\Bx})
+\frac{n(n+1)}{|\Bx|}j_n^{(1)}(k_0|\Bx|)Y_{n}^m(\hat{\Bx})\hat{\Bx},
\end{align}
where  $\mathcal{H}_n(t)=h_n^{(1)}(t)+t \left( h_n^{(1)} \right)' (t)$ and $\mathcal{J}_n(t)=j_n(t)+t  j_n ' (t)$.

The solutions to the Maxwell system can be represented as
separated variable sums of the multipole solutions, see
\cite[Section 5.3]{N}. With multipole solutions and the Helmholtz
solutions in \eqnref{vnm} and \eqnref{vtnm}, it is also possible
to expand the fundamental solution to the Helmholtz operator. For
 $k>0$, the fundamental solution $\Gamma_k$ to the Helmholtz
operator $(\Delta+k^2)$ in $\RR^3$ is \beq\label{Gk} \ds\Gamma_k
(\Bx) = -\frac{e^{ik|\Bx|}}{4 \pi |\Bx|}. \eeq Let $\mathbf{p}$ be
a fixed vector in $\RR^3$. For $|\Bx| > |\By|$, the following
addition formula holds (see \cite[Section 9.3.3]{P}):
\begin{align}
\ds\Gamma_{k}(\Bx-\By) \mathbf{p} =\nonumber&
 -\sum_{n=1}^{\infty}\frac{i{k}}{n(n+1)} \frac{\epsilon}{\mu} \sum_{m=-n}^n  \bE_{n,m}^{TM}(k;\Bx)
 \overline{ \widetilde{\bE}_{n,m}^{TM}(k;\By)}\cdot \mathbf{p} \\
\ds &+ \sum_{n=1}^{\infty}\frac{i{k}}{n(n+1)} \sum_{m=-n}^n  \bE_{n,m}^{TE}(k;\Bx)
 \overline{\widetilde{\bE}_{n,m}^{TE}(k;\By)}\cdot \mathbf{p} \nonumber\\
\ds&-\frac{i}{k}  \sum_{n=1}^{\infty} \sum_{m=-n}^n \nabla v_{n,m}
(k;\Bx) \overline{\nabla \tilde{v}_{n,m} (k;\By)} \cdot
\mathbf{p},\label{Gkexpan}
\end{align}
with $v_{n,m}$ and $\tilde{v}_{n,m}$ being defined by \eqnref{vnm}
and \eqnref{vtnm}.

Plane wave solutions to the Maxwell equations have the expansion using the multipole solutions as well (see
\cite{K}). The incoming wave $\bE^i(\Bx)=ik(\Bq\times \Bp)\times \Bq e^{ik \Bq\cdot \Bx}$, where  $\Bq \in S$ is
the direction of propagation and the vector $\Bp\in\RR^3$ is the direction of polarization, is expressed as
\begin{equation} \label{expplanewave}
 \ds \bE^i(\Bx) =\ds  ik\sum_{p=1}^{\infty} \frac{4\pi i^p}{\sqrt{p(p+1)}}
 \sum_{q=-p}^{p}   \bigg[  \bigr(\mathbf{V}_{p,q}({\Bq}) \cdot \Bc \bigr) \widetilde{\bE}_{p,q}^{TE}(\Bx)
 -\frac{1}{i \omega \mu} \bigr(\mathbf{U}_{p,q}(\Bq)
  \cdot \Bc\bigr)\widetilde{\bE}_{p,q}^{TM}(\Bx)\bigg],
\end{equation}
where $\mathbf{c}=(\Bq\times \Bp)\times \Bq$.

\section{Scattering coefficients of an inclusion}

Let $D$ be a bounded domain in $\mathbb{R}^3$ with $\mathcal{C}^{1, \Ga}$ boundary for some $\Ga>0$, and
let $(\epsilon_0,\mu_0)$ be the pair
 of electromagnetic parameters (permittivity and permeability) of $\RR^3 \setminus \overline{D}$
 and $(\epsilon_1, \mu_1)$ be that of $D$. We assume that
 $\epsilon_0,\epsilon_1,\mu_0,$ and $\mu_1$ are positive
 constants. Then the permittivity and permeability distributions are given by
\begin{equation*}
\epsilon=\epsilon_0 \chi ( \RR^3 \setminus \overline{D}) +
\epsilon_1 \chi(D) \quad \mbox{and} \quad \mu =\mu_0 \chi ( \RR^3
\setminus \overline{D}) + \mu_1 \chi(D),
\end{equation*}
 where $\chi$ denotes the
characteristic function. In the sequel, we set $k=\omega
\sqrt{\epsilon_1 \mu_1}$ and $k_0=\omega \sqrt{\epsilon_0 \mu_0}$.

For a given solution $(\bE^i,\bH^i)$ to the Maxwell equations
\begin{equation*}
 \ds\ \left \{
 \begin{array}{ll}
\ds\nabla\times{\bE^i} = i\omega\mu_0{\bH^i} \quad &\mbox{in } \mathbb{R}^3,\\
\ds\nabla\times{\bH^i} = -i\omega\epsilon_0 {\bE^i}\quad &\mbox{in } \mathbb{R}^3,
 \end{array}
 \right .
 \end{equation*}
let $(\bE,\bH)$ be the solution to the following Maxwell
equations:
\beq \label{eqn_scatt_field}
 \ds\ \left \{
 \begin{array}{ll}
\ds\nabla\times{\bE} = i\omega\mu{\bH} \quad &\mbox{in } \mathbb{R}^3,\\
\ds\nabla\times{\bH} = -i\omega\epsilon {\bE}\quad &\mbox{in } \mathbb{R}^3,\\
\ds(\bE-\bE^i, \bH-\bH^i)&\mbox{satisfies the Silver-M\"{u}ller
radiation condition}.
 \end{array}
 \right .
 \eeq
We emphasize that along the interface $\p D$, the following transmission condition holds:
\beq\label{transcond}
[\bm{\nu}\times{\bE}]=[\bm{\nu}\times{\bH}]=0 .
\eeq
Here, $[\nu\times{\bE}]$ denotes the jump of $\nu\times{\bE}$ along $\p D$, namely,
 $$[\nu\times{\bE}]=(\nu\times\bE)\bigr|^+_{\p D}-(\nu\times\bE)\bigr|^-_{\p D}.$$

Let  $\nabla_{\p D}\cdot$ denote the surface divergence. We
introduce the function space
$$
TH({\rm div}, \p D):=\Bigr\{ \bm{u} \in L_T^2(\partial D):
\nabla_{\partial D}\cdot \bm{u} \in L^2(\partial D) \Bigr\},
$$
equipped with the norm
$$\|\bm{u}\|_{TH({\rm div}, \p D)}=\|\bm{u}\|_{L^2(\p D)}+\|\nabla_{\p D}\cdot\bm{u}\|_{L^2(\p D)}.$$
For a density $\bm{\varphi} \in TH({\rm div}, \p D)$, we define
the single layer potential associated with the fundamental
solutions $\Gamma_k$ given in \eqnref{Gk} by
$$
\ds\Scal_D^{k}[\bm{\varphi}](\Bx) := \int_{\p D}
\Gamma_k(\Bx-\By) \bm{\varphi}(\By) d \sigma(\By), \quad \Bx \in \mathbb{R}^3.
$$
For a scalar density contained in $L^2(\p D)$, the single layer potential is defined by the same way.
We also define boundary integral operators:
\begin{align*}
\ds\mathcal{L}_D^k[\bm{\varphi}] (\Bx)& := \bigr(\bm{\nu} \times \bigr(k^2 \Scal_D^k[
\bm{\varphi}] + \nabla \Scal_D^k[\nabla_{\p D}\cdot \bm{\varphi}]\bigr) \Bigr)(\Bx),\\
\ds\mathcal{M}_D^k[\bm{\varphi}](\Bx) &:= \mbox{p.v.} \int_{\p D}
\bm{\nu}(\Bx) \times \Bigr( \nabla_{\Bx} \times \bigr(
\Gamma_k(\Bx-\By) \bm{\varphi} (\By)\bigr)  \Bigr) d
\sigma(\By), \quad \Bx \in \p D.
\end{align*}
 In the same way, we define $\Scal_D^{k_0}$,
$\mathcal{L}_D^{k_0}$, and $\mathcal{M}_D^{k_0}$ associated with
$\Gamma_{k_0}$ instead of $\Gamma_k$. Then the solution to
\eqnref{eqn_scatt_field} can be represented as the following: \beq
\label{represent} \ds\bE (\Bx)= \left \{
 \begin{array}{ll}
\ds \bE^i(\Bx) + \mu_0 \nabla \times \Scal_D^{k_0} [\bm{\varphi}](\Bx) +
\nabla\times\nabla\times\Scal_D^{k_0} [\bm{\psi}](\Bx) ,\quad &\Bx \in \mathbb{R}^3 \setminus \overline{D},\\
\nm \ds\mu_1 \nabla \times \Scal_D^{k} [\bm{\varphi}](\Bx) +
\nabla\times\nabla\times\Scal_D^{k} [\bm{\psi}](\Bx) ,\quad &\Bx
\in D,
 \end{array}
 \right .
 \eeq
and
$$ \bH(\Bx) = -\frac{i}{\omega \mu}\bigr(\nabla \times \bE\bigr)(\Bx),\quad \Bx \in \RR^3\setminus \p D,
$$
where the pair $(\bm{\varphi}, \bm{\psi}) \in TH({\rm div}, \p D)
\times TH({\rm div}, \p D)$ is the unique solution to \beq
\label{phi_psi}
\begin{bmatrix}
\ds\frac{\mu_1+\mu_0}{2}I + \mu_1 \mathcal{M}_D^k -\mu_0
\mathcal{M}_D^{k_0} &
\ds\mathcal{L}_D^k - \mathcal{L}_D^{k_0} \\
\ds\mathcal{L}_D^k - \mathcal{L}_D^{k_0} & \ds \left( \frac{k^2}{2
\mu_1} + \frac{k_0^2}{2 \mu_0}\right)I +
\frac{k^2}{\mu_1}\mathcal{M}_D^k -
\frac{k_0^2}{\mu_0}\mathcal{M}_D^{k_0}
\end{bmatrix}
\begin{bmatrix}
\bm{\varphi} \\ \bm{\psi}
\end{bmatrix}
=
\left.\begin{bmatrix}
\bE^i \times \bm{\nu}\\
i \omega \bH^i \times \bm{\nu}
\end{bmatrix}\right|_{\p D}.
\eeq The invertibility of the system of equations (\ref{phi_psi})
on $TH({\rm div}, \p D) \times TH({\rm div}, \p D)$ was proved in
\cite{T}. Moreover, there exists a constant $C=C(\epsilon, \mu,
\omega)$ such that \beq \label{phi_psi_Ein_Hin} \ds\| \bm{\varphi}
\|_{TH({\rm div},\p D)}+ \| \bm{\psi} \|_{TH({\rm div},\p D)} \leq
C\bigr(\| \bE^i \times \bm{\nu} \|_{TH({\rm div},\p D)}+ \|\bH^i
\times \bm{\nu} \|_{TH({\rm div},\p D)}\bigr). \eeq

From \eqnref{Gkexpan} (with $k_0$ in the place of $k$) and
\eqnref{represent} it follows that, for sufficiently large
$|\Bx|$,
\begin{equation}\label{Esexpan}
\ds(\bE-\bE^i)(\Bx)=\sum_{n=1}^{\infty}\frac{i{k_0}}{n(n+1)} \sum_{m=-n}^n  \Bigr(\alpha_{n,m}\bE_{n,m}^{TE}(k_0;\Bx)\ + \ \beta_{n,m}\bE_{n,m}^{TM}(k_0;\Bx)\Bigr),
\end{equation}
where
\begin{align*}
\alpha_{n,m} &= -i \omega \epsilon_0 \mu_0 \int_{\p D} \overline{
\widetilde{\bE}_{n,m}^{TM}}(k_0;\By) \cdot \bm{\varphi}(\By) + k_0^2
\int_{\p D} \overline{\widetilde{\bE}_{n,m}^{TE}}(k_0;\By) \cdot
\bm{\psi}(\By), \\
\beta_{n,m} &=  -i \omega \epsilon_0 \mu_0 \int_{\p D} \overline{ \widetilde{\bE}_{n,m}^{TE}}(k_0;\By) \cdot \bm{\varphi}(\By)
-\omega^2 \epsilon_0^2 \int_{\p D}
\overline{\widetilde{\bE}_{n,m}^{TM}}(k_0;\By) \cdot \bm{\psi}(\By).
\end{align*}

\begin{definition}\label{defW}
Let $(\bm{\varphi}_{p,q}^{TE},\bm{\psi}_{p,q}^{TE})$ be the
solution to \eqnref{phi_psi} when $\bE^i=\widetilde{\bE}_{p,q}^{TE}(k_0; \By)$ and $\bH^i= \widetilde{\bH}_{p,q}^{TE}(k_0; \By)$, and
$(\bm{\varphi}_{p,q}^{TM},\bm{\psi}_{p,q}^{TM})$ when $\bE^i=\widetilde{\bE}_{p,q}^{TM}(k_0; \By)$ and $\bH^i= \widetilde{\bH}_{p,q}^{TM}(k_0; \By)$.
The scattering coefficients
$\Bigr(W_{(n,m)(p,q)}^{TE,TE},W_{(n,m)(p,q)}^{TE,TM},
W_{(n,m)(p,q)}^{TM,TE},W_{(n,m)(p,q)}^{TM,TM}\Bigr)$ associated
with the permittivity and the permeability distributions
$\epsilon, \mu$ and the frequency $\omega$ (or $k,k_0,D$) is defined to be
\begin{align*}
\ds W_{(n,m)(p,q)}^{TE,TE}&=
 -i \omega \epsilon_0 \mu_0 \int_{\p D} \overline{ \widetilde{\bE}_{n,m}^{TM}}(k_0;\By) \cdot \bm{\varphi}_{p,q}^{TE}(\By)\ d\sigma(\By)
+ k_0^2 \int_{\p D} \overline{\widetilde{\bE}_{n,m}^{TE}}(k_0;\By)
\cdot \bm{\psi}_{p,q}^{TE}(\By)\ d\sigma(\By), \\\nonumber \ds
W_{(n,m)(p,q)}^{TE,TM}&=
 -i \omega \epsilon_0 \mu_0 \int_{\p D} \overline{ \widetilde{\bE}_{n,m}^{TM}}(k_0;\By) \cdot \bm{\varphi}_{p,q}^{TM}(\By)\ d\sigma(\By)
+ k_0^2 \int_{\p D} \overline{\widetilde{\bE}_{n,m}^{TE}}(k_0;\By)
\cdot \bm{\psi}_{p,q}^{TM}(\By)\ d\sigma(\By), \\\nonumber \ds
W_{(n,m)(p,q)}^{TM,TE}&=
 -i \omega \epsilon_0 \mu_0 \int_{\p D} \overline{\widetilde{\bE}_{n,m}^{TE}}(k_0;\By) \cdot \bm{\varphi}_{p,q}^{TE}(\By)\ d\sigma(\By)
-\omega^2 \epsilon_0^2 \int_{\p D} \overline{
\widetilde{\bE}_{n,m}^{TM}}(k_0;\By) \cdot \bm{\psi}_{p,q}^{TE}(\By)\
d\sigma(\By), \\\nonumber \ds W_{(n,m)(p,q)}^{TM,TM}&=
 -i \omega \epsilon_0 \mu_0 \int_{\p D} \overline{\widetilde{\bE}_{n,m}^{TE}}(k_0;\By) \cdot \bm{\varphi}_{p,q}^{TM}(\By)\ d\sigma(\By)
-\omega^2 \epsilon_0^2 \int_{\p D} \overline{
\widetilde{\bE}_{n,m}^{TM}}(k_0;\By) \cdot
\bm{\psi}_{p,q}^{TM}(\By)\ d\sigma(\By).
\end{align*}
\end{definition}

As we see it now, the scattering coefficients appear naturally in the expansion of the scattering amplitude.
We first obtain the following estimates of the scattering coefficients.
\begin{lemma}\label{Westimate}
There exists a constant $C$ depending on $(\epsilon, \mu, \omega)$ such that
\beq\label{W_decay}
\left|W_{(n,m)(p,q)}^{TE,TE} [\epsilon, \mu, \omega] \right|
\leq \frac{C^{n+p}}{n^n p^p},
\eeq
for all $n,m,p,q \in \mathbb{N}$. The same estimates hold for $W_{(n,m)(p,q)}^{TE,TM}$, $W_{(n,m)(p,q)}^{TM,TE}$, and $W_{(n,m)(p,q)}^{TM,TM}$.
\end{lemma}
\begin{proof}
Let $(\bm{\varphi},\bm{\psi})$ be the solution to \eqnref{phi_psi}
with $\bE^i(\By) = \widetilde{\bE}_{p,q}^{TE}(k_0;\By)$ and $\bH^i =
-\frac{i}{\omega \mu_0}\nabla \times \bE^i$. Recall that the spherical Bessel function $j_p$ behaves as
$$
j_p(t)=\frac{t^p}{1\cdot 3\cdots (2p+1)}\Bigr(1+O\left(\frac{1}{p}\right)\Bigr)\quad \mbox{as }p\rightarrow\infty,
$$
uniformly on compact subsets of $\RR$. Using Stirling's formula $p
!=\sqrt{2\pi p}(p/e)^p(1+o(1))$, we have
\beq\label{bessel_j_large_n} j_p(t) =
O\left(\frac{C^pt^p}{p^p}\right) \quad \mbox{as } p \rightarrow
\infty, \eeq uniformly on compact subset of $\RR$ with a constant
$C$ independent of $p$. Thus we have
$$
\big\| \bE^i\big\|_{TH({\rm div},\p D)}+
\big\|\bH^i\big\|_{TH({\rm div},\p D)} \leq \frac{{C'}^p}{p^p}
$$
for some constant $C'$. It then follows from \eqnref{phi_psi_Ein_Hin} that
$$
\big\| \bm{\varphi} \big\|_{L^2(\p D)} + \big\|
\bm{\psi}\big\|_{L^2(\p D)} \leq \frac{C^p}{p^p}
$$
for another constant $C$. So we get \eqnref{W_decay} from the definition of
the scattering coefficients.
\end{proof}

Suppose that the incoming wave is of the form
\beq\label{general_incident}
\bE^i(\Bx)= \sum_{p=1}^{\infty}
\sum_{q=-p}^p \Bigr(a_{p,q} \widetilde{\bE}_{p,q}^{TE}(k_0;\Bx) +
b_{p,q} \widetilde{\bE}_{p,q}^{TM}(k_0;\Bx)\Bigr)
\eeq
for some constants $a_{p,q}$ and $b_{p,q}$. Then the solution $(\bm{\varphi},\bm{\psi})$ to
\eqnref{phi_psi} is given by
\begin{align*}
\bm{\varphi} &=\sum_{p=1}^{\infty} \sum_{q=-p}^p \Bigr(a_{p,q}
\bm{\varphi}_{p,q}^{TE} + b_{p,q} \bm{\varphi}_{p,q}^{TM}\Bigr), \\
\bm{\psi} &=\sum_{p=1}^{\infty} \sum_{q=-p}^p \Bigr( a_{p,q}
\bm{\psi}_{p,q}^{TE} + b_{p,q} \bm{\psi}_{p,q}^{TM}\Bigr).
\end{align*}
By \eqnref{Esexpan} and Definition \ref{defW}, the solution
$\bE$ to \eqnref{eqn_scatt_field} can be represented as
\begin{equation}\label{generalexpan}
\ds(\bE-\bE^i)(\Bx)=\sum_{n=1}^{\infty}\frac{i{k_0}}{n(n+1)} \sum_{m=-n}^n
\Bigr(\alpha_{n,m}\bE_{n,m}^{TE}(k_0;\Bx)\ + \ \beta_{n,m}\bE_{n,m}^{TM}(k_0;\Bx)\Bigr), \quad |\Bx| \to \infty,
\end{equation}
where
\begin{equation}\label{alphabeta}
\begin{cases}
\ds\alpha_{n,m} = \sum_{p=1}^{\infty} \sum_{q=-p}^p \left(a_{p,q} W_{(n,m)(p,q)}^{TE,TE}+b_{p,q} W_{(n,m)(p,q)}^{TE,TM}\right),\\
 \ds\beta_{n,m} = \sum_{p=1}^{\infty} \sum_{q=-p}^p \left(a_{p,q} W_{(n,m)(p,q)}^{TM,TE}+b_{p,q} W_{(n,m)(p,q)}^{TM,TM}\right).
\end{cases}
\end{equation}

Using \eqnref{generalexpan}, \eqnref{alphabeta} and the behavior
of the spherical Bessel functions, we can estimate the far-field
pattern of the scattered wave $(\bE-\bE^i)$. The far-field pattern (also called the scattering amplitude) $\mathbf{A}_{\infty}[\epsilon, \mu,\omega]$ is defined by
\beq
\bE(\Bx)-\bE^i(\Bx) = \frac{e^{i k_0 |\Bx|}}{k_0 |\Bx|}
\mathbf{A}_{\infty}[\epsilon,\mu,\omega] (\hat{\Bx}) +
o(|\Bx|^{-1}) \quad \mbox{as }|\Bx| \rightarrow \infty.
\eeq

Since the spherical Bessel function $ h_n^{(1)}$ behaves like
\begin{equation*}
\begin{cases}
\ds h_n^{(1)} (t) \sim \frac{1}{t} e^{it} e^{-i\frac{n+1}{2} \pi} \quad &\mbox{as } t \rightarrow \infty,\\
\ds (h_n^{(1)})'(t) \sim \frac{1}{t} e^{it} e^{-i\frac{n}{2} \pi}\quad &\mbox{as } t \rightarrow \infty,
\end{cases}
\end{equation*}
one can easily see by using \eqnref{curlE} that
$$
\begin{cases}
\ds \bE_{n,m}^{TE}(k_0;\Bx) \sim \frac{e^{i k_0 |\Bx|}}{k_0 |\Bx|} e^{-i\frac{n+1}{2}\pi}\ \bigr( -\sqrt{n(n+1)}\bigr)  \mathbf{V}_{n,m}(\hat{\Bx}) \quad &\mbox{as } |\Bx| \rightarrow \infty ,\\
\ds \bE_{n,m}^{TM}(k_0;\Bx) \sim \frac{e^{i k_0 |\Bx|}}{k_0 |\Bx|} \sqrt{\frac{\mu_0}{\epsilon_0}}e^{-i\frac{n+1}{2} \pi}\bigr(-\sqrt{n(n+1)}\bigr) \mathbf{U}_{n,m}(\hat{\Bx}) \quad &\mbox{as } |\Bx| \rightarrow \infty.
\end{cases}
$$

The following result holds.
\begin{prop}
If $\bE^i$ is given by \eqnref{general_incident}, then the
corresponding scattering amplitude can be expanded as
\begin{align} \label{aexpansion}
\ds \mathbf{A}_{\infty}[\epsilon, \mu, \omega] (\hat{\Bx})=&\sum_{n=1}^{\infty}\frac{-i^{-n}{k_0}}{\sqrt{n(n+1)}}
 \sum_{m=-n}^n\Bigr(\alpha_{n,m} \mathbf{V}_{n, m}(\hat{\Bx})
+\beta_{n,m} \sqrt{\frac{\mu_0}{\epsilon_0}}\mathbf{U}_{n,m}(\hat{\Bx}) \Bigr),
\end{align}
where $\alpha_{n,m}$ and $\beta_{n,m}$ are defined by
\eqnref{alphabeta}.
\end{prop}

We emphasize that since $ \{ \mathbf{V}_{n,m}, \mathbf{U}_{n,m} \}$ forms an orthogonal basis of $L^2_T(S)$, the
conversion of the far-field to the near field is achieved via
formula \eqnref{generalexpan}.

We now consider the case where the incident wave $\bE^i$ is given
by a plane wave $e^{i \mathbf{k}\cdot \Bx}\mathbf{c}$ with
$|\mathbf{k}|=k_0$ and $\mathbf{k}\cdot \mathbf{c}=0$. It follows from
\eqnref{expplanewave} that
$$
  e^{i \mathbf{k} \cdot \Bx}\mathbf{c} =\sum_{p=1}^{\infty} \frac{4 \pi i^p}{\sqrt{p(p+1)}} \sum_{q=-p}^{p} \left[  \bigr(\mathbf{V}_{p,q}(\hat{\mathbf{k}}) \cdot \mathbf{c} \bigr) \widetilde{\bE}_{p,q}^{TE}(k_0;\Bx)
-\frac{1}{i \omega \mu_0}
\bigr(\mathbf{U}_{p,q}(\hat{\mathbf{k}})
\cdot \mathbf{c}\bigr) \widetilde{\bE}_{p,q}^{TM}(k_0;\Bx)\right],$$ where
$\hat{\mathbf{k}} = \mathbf{k}/k_0 \in S$, and therefore,
$$a_{p,q}=  \frac{4 \pi i^p}{\sqrt{p(p+1)}}(\mathbf{V}_{p,q}(\hat{\mathbf{k}}) \cdot \mathbf{c} )\quad \mbox{and} \quad
b_{p,q} = -\frac{4 \pi i^p}{\sqrt{p(p+1)}}\frac{1}{i \omega \mu_0}
(\mathbf{U}_{p,q}(\hat{\mathbf{k}})
\cdot \mathbf{c}) .$$
Hence, the scattering amplitude, which we denote by $\mathbf{A}_{\infty}[\epsilon, \mu, \omega](\bm{c},\hat{\bm{k}};\hat{\bm{x}})$, is given by \eqnref{aexpansion} with the coefficients $\alpha_{n,m}$ and $\beta_{n,m}$
\begin{equation}\label{far_field}
\begin{cases}
\ds\alpha_{n,m} = \sum_{p=1}^{\infty} \sum_{q=-p}^p\frac{4 \pi i^p}{\sqrt{p(p+1)}} \left[ (\mathbf{V}_{p,q}(\hat{\mathbf{k}}) \cdot \mathbf{c} )W_{(n,m)(p,q)}^{TE,TE}-\frac{1}{i \omega \mu_0} (\mathbf{U}_{p,q}(\hat{\mathbf{k}}) \cdot \mathbf{c}) W_{(n,m)(p,q)}^{TE,TM}\right],\\
 \ds\beta_{n,m} = \sum_{p=1}^{\infty} \sum_{q=-p}^p \frac{4 \pi i^p}{\sqrt{p(p+1)}}\left[ (\mathbf{V}_{p,q}(\hat{\mathbf{k}}) \cdot \mathbf{c} ) W_{(n,m)(p,q)}^{TM,TE}-\frac{1}{i \omega \mu_0} (\mathbf{U}_{p,q}(\hat{\mathbf{k}})  \cdot \mathbf{c}) W_{(n,m)(p,q)}^{TM,TM}\right].
\end{cases}
\end{equation}
These formulas tell us that the scattering coefficients appear in
the expansion of the scattering amplitude.

We now investigate the low frequency behavior of the scattering coefficients. Let $\Gamma(\Bx) := - 1/(4\pi |\Bx|)$ denote the fundamental
solution corresponding to the case $k=0$, and $\mathcal{M}_D$
the associated boundary integral operator:
$$
\ds\mathcal{M}_D[\bm{\varphi}](\Bx):= \mbox{p.v.} \int_{\p D}
\bm{\nu}(\Bx) \times \Bigr( \nabla_{\Bx} \times \bigr(
\Gamma(\Bx-\By) \bm{\varphi} (\By)\bigr)  \Bigr) d \sigma(\By),
\quad \bm{\varphi} \in TH({\rm div}, \p D).
$$
Analogously to \eqnref{phi_psi}, one can prove that there is a
unique solution  $(\bm{\varphi}^{(0)}, \bm{\psi}^{(0)}) \in
TH({\rm div}, \p D) \times TH({\rm div}, \p D)$ to the following
equations:
\beq \label{phi_psi0}
\begin{bmatrix}
\ds (\mu_1  -\mu_0) \bigg( \frac{\mu_1+\mu_0}{2(\mu_1  -\mu_0)}I +
\mathcal{M}_D \bigg)&
0 \\
0 & \ds (\epsilon_1  -\epsilon_0) \bigg(
\frac{\epsilon_1+\epsilon_0}{2(\epsilon_1  -\epsilon_0)}I +\mathcal{M}_D \bigg)
\end{bmatrix}
\begin{bmatrix}
\bm{\varphi}^{(0)} \\ \omega\bm{\psi}^{(0)}
\end{bmatrix}
= \left.\begin{bmatrix}
\bE^i \times \bm{\nu}\\
i  \bH^i \times \bm{\nu}
\end{bmatrix}\right|_{\p D}.
\eeq
In fact, since $\p D$ is $\mathcal{C}^{1, \Ga}$, $\mathcal{M}_D$ is compact and we may apply the Fredholm
alternative to prove unique solvability of above equation. Moreover, we have
\beq
\ds\| \bm{\varphi}^{(0)} \|_{TH({\rm
div},\p D)}+ \omega\| \bm{\psi}^{(0)} \|_{TH({\rm div}, \p D)} \leq C(\| \bE^i\times\bm{\nu} \|_{TH({\rm div}, \p D)}+ \|\bH^i\times\bm{\nu}
\|_{TH({\rm div}, \p D)}), \eeq
with a constant $C=C(\epsilon,\mu)$.

Let $\rho$ be a small positive number and consider the boundary integral equation \eqnref{phi_psi} with $k$, $k_0$, and $\omega$ replaced by $\rho k$, $\rho k_0$, and $\rho\omega$, respectively.  Then, we have (see \cite{griesmaier})
$${\mathcal{M}}^{\rho k}_D - {\mathcal{M}}_D = O(\rho^2),\quad {\mathcal{M}}^{\rho k_0}_D - {\mathcal{M}}_D = O(\rho^2),$$
and
$${\mathcal{L}}^{\rho k}_D - {\mathcal{L}}_D^{\rho k_0} = O(\rho^2).$$
Since
$$
\left( \frac{k^2}{2
\mu_1} + \frac{k_0^2}{2 \mu_0}\right)I +
\frac{k^2}{\mu_1}\mathcal{M}_D^k -
\frac{k_0^2}{\mu_0}\mathcal{M}_D^{k_0} = \rho^2 \omega^2 \bigg[
\frac{\epsilon_1+\epsilon_0}{2}I + (\epsilon_1  -\epsilon_0) \mathcal{M}_D + O(\rho^2) \bigg],
$$
if we express the solution $(\bm{\varphi},\bm{\psi})$ to \eqnref{phi_psi} as $(\bm{\varphi},\bm{\psi}):=(\bm{\varphi}^\rho, \rho \omega \bm{\psi}^\rho)$, then it satisfies
$$\Bigr(A+O(\rho)\Bigr)
\begin{bmatrix}
\bm{\varphi}^\rho \\ {\rho\omega}\bm{\psi}^\rho
\end{bmatrix}
=\left.\begin{bmatrix}
\bE^i \times \bm{\nu}\\
i  \bH^i \times \bm{\nu}
\end{bmatrix}\right|_{\p D},$$
where $A$ is the 2-by-2 matrix appeared in the left-hand side of \eqnref{phi_psi0}. From the invertibility of $A$,
it follows that there are constants
$\rho_0$ and $C=C(\epsilon, \mu, \omega)$ independent of $\rho$ as
long as $ \rho \leq \rho_0$ such that \beq
\label{phi_psi_Ein_Hin_rho} \ds\| \bm{\varphi}^{\rho} \|_{TH({\rm
div},\p D)}+ \rho\omega\| \bm{\psi}^{\rho} \|_{TH({\rm div}, \p D)}
\leq C\bigr(\| \bE^i\times\bm{\nu} \|_{TH({\rm div}, \p D)}+ \|\bH^i\times\bm{\nu}
\|_{TH({\rm div}, \p D)}\bigr). \eeq

\begin{lemma}\label{Westimate2}
There exists $\rho_0$ such that, for all $\rho \leq \rho_0$,
\begin{align}\label{W_decay_rho}
\ds&\left|W_{(n,m)(p,q)}^{TE,TE} [\epsilon, \mu, \rho\omega] \right| \leq \frac{C^{n+p}}{n^n p^p} \rho^{n+p+1},%
\end{align}
for all $n,m,p,q \in \mathbb{N}$, where the constant $C$ depends
on $(\epsilon, \mu, \omega)$ but is independent of $\rho$. The same estimate holds for $W_{(n,m)(p,q)}^{TE,TM}$, $W_{(n,m)(p,q)}^{TM,TE}$, and $W_{(n,m)(p,q)}^{TM,TM}$.
\end{lemma}
\begin{proof}
Let $(\bm{\varphi},\bm{\psi})$ be the solution to \eqnref{phi_psi}
with $\bE^i(\By) = \widetilde{\bE}_{p,q}^{TE}(\rho k_0;\By)$ and $\bH^i =
-\frac{i}{\rho \omega \mu_0}\nabla \times \bE^i$.
Then, from \eqnref{bessel_j_large_n}, it follows that
$$
\big\| \bE^{i, \rho} \big\|_{TH({\rm div},\p D)}+ \big\| \bH^{i,
\rho} \big\|_{TH({\rm div},\p D)} \leq \frac{C^p}{p^p} \rho^p,
$$
where $C$ is independent of $\rho$, and hence
$$
\big\| \bm{\varphi}^{\rho}\big\|_{L^2(\p D)} + \rho \big\|
\bm{\psi}^{\rho} \big\|_{L^2(\p D)} \leq \frac{C^p}{p^p} \rho^{p},
$$
for $\rho \leq \rho_0$ for some $\rho_0$.
So we get \eqnref{W_decay_rho} from the definition of the scattering coefficients in Definition \ref{defW}.
\end{proof}

\section{S-vanishing structures}\label{mlstructure}

The purpose of this section is to construct multilayered structures whose scattering coefficients vanish,
which we call {\it
S-vanishing structures}.  The multi-layered structure is defined as follows: For positive numbers
$r_1,\dotsc, r_{L+1} $ with $2=r_1 > r_2 > \dotsb r_{L+1} = 1$, let
$$
A_j := \{ \Bx : r_{j+1} \leq |\Bx| < r_j \}, \quad j=1, \dotsc, L, \quad A_0 :=\mathbb{R}^2 \setminus \overline{A_1},
\quad
A_{L+1}(=D) := \{ \Bx : |\Bx| < 1 \},
$$
and
$$
\GG_j= \{ |\Bx|=r_j \}, \quad j=1, \ldots, L+1.
$$
Let $( \mu_j, \epsilon_j )$ be the pair of permeability and
permittivity parameters of $A_j$ for $j=1, \dotsc, L+1$. Set $
\mu_0 = 1$ and $ \epsilon_0 = 1$. We then define
\beq
\mu=\sum_{j=0}^{L+1} \mu_j \chi (A_j) \quad \mbox{and} \quad
\epsilon = \sum_{j=0}^{L+1} \epsilon_j \chi (A_j),
\eeq
which are permeability and permittivity distributions of the layered structure.

The scattering coefficients $\Bigr(
W_{(n,m)(p,q)}^{TE,TE},W_{(n,m)(p,q)}^{TE,TM},W_{(n,m)(p,q)}^{TM,TE},W_{(n,m)(p,q)}^{TM,TM}\Bigr)$
are defined as before, namely, if $\bE^i$ given as in
\eqnref{general_incident}, the scattered field $\bE - \bE^i$ can be expanded as \eqnref{generalexpan}
and \eqnref{alphabeta}. The transmission condition on each interface $\GG_j$ is given by
\beq\label{transcond2}
[\hat{\Bx}\times{\bE}]=[\hat{\Bx}\times{\bH}]=0.
\eeq
We assume that the core $A_{L+1}$ is perfectly conducting (PEC), namely,
\beq
\bE \times \bm{\nu} =0 \quad\mbox{on } \GG_{L+1}=\p A_{L+1}.
\eeq

Thanks to the symmetry of the layered (radial) structure, the scattering coefficients are much simpler than the general case. In fact, if the
incident field is given by $\bE^i =\widetilde{\bE}_{n,m}^{TE}$, then the solution $\bE$ to \eqnref{eqn_scatt_field} takes the form
\beq \label{E_n_TE}
\bE(\Bx)=\tilde{a}_j \widetilde{\bE}_{n,m}^{TE}(\Bx) + a_j
\bE_{n,m}^{TE}(\Bx),\quad \Bx \in A_j, \quad j=0,\ldots ,L, \eeq
with $\tilde{a}_0=1$. From \eqnref{curlE} and \eqnref{curlE2}, the interface condition \eqnref{transcond2} amounts to
\begin{align} \label{recurrence_TE}
&\begin{bmatrix}
\ds j_n(k_{j}r_j)& \ds h_n^{(1)}(k_{j}r_j)\\
\ds \frac{1}{\mu_{j}}  \mathcal{J}_n (k_{j} r_j) &
 \ds\frac{1} {\mu_{j}}\mathcal{H}_n (k_{j} r_j)
\end{bmatrix}
\begin{bmatrix}
\ds \tilde{a}_j \\ \ds a_j
\end{bmatrix} \nonumber\\
&=\
\begin{bmatrix}  \ds j_n(k_{j-1}r_j)&\ds h_n^{(1)}(k_{j-1}r_j)\\
\ds\frac{1}{\mu_{j-1}}  \mathcal{J}_n (k_{j-1} r_j) &
\ds\frac{1} {\mu_{j-1}}\mathcal{H}_n (k_{j-1} r_j)\end{bmatrix}
\begin{bmatrix}
\ds\tilde{a}_{j-1} \\ \ds a_{j-1}
\end{bmatrix}, \quad j=1, \ldots, L,
\end{align}
where $\mathcal{H}_n(t)=h_n^{(1)}(t)+t \left( h_n^{(1)} \right)' (t)$ and $\mathcal{J}_n(t)=j_n(t)+t  j_n ' (t)$,
and the PEC boundary condition on $\GG_{L+1}$ is
\beq \label{PEC_bc_TE}
\begin{bmatrix}
j_n(k_{L})&h_n^{(1)}(k_{L})\\
0&0
\end{bmatrix}
\begin{bmatrix}
\tilde{a}_L \\ a_L
\end{bmatrix}
=
\begin{bmatrix}
0 \\0
\end{bmatrix}.
\eeq
Since the matrices appeared in \eqnref{recurrence_TE} are invertible, one can see that there are $a_j$ and $\tilde{a}_j$, $j=0, 1, \ldots L$ satisfying \eqnref{recurrence_TE} and \eqnref{PEC_bc_TE}. Similarly, one can see that if the incident field is given by $\bE^i= \widetilde{\bE}_{n,m}^{TM}(\Bx)$, then the solution  $\bE$ takes the form
\begin{equation}\label{E_n_TM}
\ds \bE(\Bx)=\tilde{b}_j \widetilde{\bE}_{n,m}^{TM}(\Bx) + b_j \bE_{n,m}^{TM}(\Bx),\quad \Bx \in A_j, \quad j=0, 1,. . . ,L
\end{equation}
for some constants $b_j$ and $\tilde{b}_j$ ($\tilde{b}_0=1$). One can see now from \eqnref{E_n_TE} and \eqnref{E_n_TM} that
the scattering coefficients satisfy
\begin{align*}
\ds W_{(n,m)(p,q)}^{TE,TM} &= W_{(n,m)(p,q)}^{TM,TE} = 0 \quad \mbox{for all } (m,n) \mbox{ and } (p,q),  \\
\ds W_{(n,m)(p,q)}^{TE,TE} &= W_{(n,m)(p,q)}^{TM,TM} = 0 \quad \mbox{if } (m,n) \neq (p,q),
\end{align*}
and, since \eqnref{E_n_TE} and \eqnref{E_n_TM} hold independently of $m$, we have
\begin{align*}
\ds W_{(n,0)(n,0)}^{TE,TE} &= W_{(n,m)(n,m)}^{TE,TE} ,  \\
\ds W_{(n,0)(n,0)}^{TM,TM} &= W_{(n,m)(n,m)}^{TM,TM} \quad\quad\  \mbox{for } -n \leq m \leq n.
\end{align*}
Moreover, if we write
$$
W_n^{TE} := W_{(n,0)(n,0)}^{TE} \quad \mbox{and} \quad W_n^{TM} := W_{(n,0)(n,0)}^{TM},
$$
then we have
\begin{equation}\label{Wnteall}
W_n^{TE}=-\frac{in(n+1)}{k_0}a_0 \quad\mbox{and}\quad W_n^{TE}=-\frac{in(n+1)}{k_0}b_0.
\end{equation}

Suppose now that $\widetilde{\bE}_{n,0}^{TE}$ is the incident field and the solution $\bE$ is given by
$$
\bE(\Bx)=\tilde{a}_j \widetilde{\bE}_{n,0}^{TE}(\Bx) + a_j
\bE_{n,0}^{TE}(\Bx),\quad \Bx \in A_j, \quad j=0,\ldots ,L,
$$
with $\tilde{a}_0=1$, where
the coefficients $\tilde{a}_j$'s and $a_j$'s are determined by \eqnref{recurrence_TE} and \eqnref{PEC_bc_TE}.
We have from \eqnref{recurrence_TE} that
$$
\begin{bmatrix}
\ds \tilde{a}_j \\ \ds a_j
\end{bmatrix} = \begin{bmatrix}
\ds j_n(k_{j}r_j)& \ds h_n^{(1)}(k_{j}r_j)\\
\ds \frac{1}{\mu_{j}}  \mathcal{J}_n (k_{j} r_j) &
 \ds\frac{1} {\mu_{j}}\mathcal{H}_n (k_{j} r_j)
\end{bmatrix}^{-1}
\begin{bmatrix}  \ds j_n(k_{j-1}r_j)&\ds h_n^{(1)}(k_{j-1}r_j)\\
\ds\frac{1}{\mu_{j-1}}  \mathcal{J}_n (k_{j-1} r_j) &
\ds\frac{1} {\mu_{j-1}}\mathcal{H}_n (k_{j-1} r_j)\end{bmatrix}
\begin{bmatrix}
\ds\tilde{a}_{j-1} \\ \ds a_{j-1}
\end{bmatrix},
$$
for $j=1, \ldots, L$. Substituting these relations into \eqnref{PEC_bc_TE} yields
\beq \label{00_P_TE}
\begin{bmatrix}
0 \\ 0
\end{bmatrix}
=
P_{n}^{TE}[\ep,\mu,\omega]
\begin{bmatrix}
\tilde{a}_0\\ a_0
\end{bmatrix},
\eeq
where
\begin{align}\nonumber
&P_{n}^{TE}[\ep,\mu,\omega] :=
\begin{bmatrix}
p_{n,1}^{TE} & p_{n,2}^{TE} \\ 0 & 0
\end{bmatrix}
= (-i \omega)^L \left( \prod_{j=1}^{L} \mu_j^{\frac{3}{2}} \ep_j^{\frac{1}{2}} r_j \right)
\begin{bmatrix}
j_n(k_{L})&h_n^{(1)}(k_{L})\\
0&0
\end{bmatrix}
\\\label{eqn:PTE}
&\qquad\qquad\times \prod_{j=1}^{L}
\begin{bmatrix}
\ds\frac{1} {\mu_{j}}\mathcal{H}_n (k_{j} r_j) & \ds -h_n^{(1)}(k_{j}r_j) \\
\ds-\frac{1}{\mu_{j}}  \mathcal{J}_n (k_{j} r_j) & \ds j_n(k_{j}r_j)
\end{bmatrix}
\begin{bmatrix}\ds j_n(k_{j-1}r_j) & \ds h_n^{(1)}(k_{j-1}r_j)\\\ds\frac{1}{\mu_{j-1}}  \mathcal{J}_n (k_{j-1} r_j)
&\ds\frac{1} {\mu_{j-1}}\mathcal{H}_n (k_{j-1} r_j)\end{bmatrix}.
\end{align}
We then have from \eqnref{00_P_TE}
\begin{equation}\label{Wnte}
W_n^{TE}=-\frac{in(n+1)}{k_0}a_0 =-\frac{in(n+1)}{k_0}\frac{p^{TE}_{n,1}}{p^{TE}_{n,2}}.
\end{equation}

\smallskip

 Similarly, for $W_n^{TM}$, we look for another solution  $\bE$ of the form
\begin{equation*}
\ds \bE(\Bx)=\tilde{b}_j \widetilde{\bE}_{n,0}^{TM}(\Bx) + b_j \bE_{n,0}^{TM}(\Bx),\quad \Bx \in A_j, \quad j=0,. . . ,L,
\end{equation*}
with $\tilde{b}_0=1$. The transmission conditions become
\begin{align} \label{recurrence_TM}
&\begin{bmatrix}
\ds\frac{1}{\ep_{j}} \mathcal{J}_n (k_{j} r_j)&
\ds\frac{1} {\ep_{j}}  \mathcal{H}_n (k_{j} r_j)\\
\ds j_n(k_{j}r_j)& \ds h_n^{(1)}(k_{j}r_j)\end{bmatrix}
\begin{bmatrix}
\tilde{b}_j \\ b_j
\end{bmatrix} \nonumber \\
& =\
\begin{bmatrix}
\ds\frac{1}{\ep_{j-1}} \mathcal{J}_n (k_{j-1} r_j)&
\ds\frac{1} {\ep_{j-1}} \mathcal{H}_n (k_{j-1} r_j)\\\
\ds j_n(k_{j-1}r_j)& \ds h_n^{(1)}(k_{j-1}r_j)
\end{bmatrix}
\begin{bmatrix}
\tilde{b}_{j-1} \\ b_{j-1}
\end{bmatrix},\quad j=1, \ldots, N+1,
\end{align}
and the PEC boundary condition on the inner most layer is
\beq \label{PEC_bc_TM}
\begin{bmatrix}
 \mathcal{J}_n (k_{L} )&
 \mathcal{H}_n (k_{L} ) \\
0&0
\end{bmatrix}
\begin{bmatrix}
\tilde{b}_L\\ b_L
\end{bmatrix}
=
\begin{bmatrix}
0 \\0
\end{bmatrix}.
\eeq
From \eqnref{recurrence_TM} and \eqnref{PEC_bc_TM}, we obtain
\beq \label{00_P_TM}
\begin{bmatrix}
0 \\ 0
\end{bmatrix}
=
P_{n}^{TM}[\ep,\mu,\omega]
\begin{bmatrix}
\tilde{b}_0 \\ b_0
\end{bmatrix},
\eeq
where
\begin{align}
&\ds P_{n}^{TM}[\ep,\mu,\omega] :=
\begin{bmatrix}
\ds p_{n,1}^{TM} & \ds p_{n,2}^{TM} \\ 0 & 0
\end{bmatrix}
\ds= (i \omega)^L \left( \prod_{j=1}^{L} \mu_j^{\frac{1}{2}} \ep_j^{\frac{3}{2}} r_j \right)
\begin{bmatrix}
\ds\mathcal{J}_n(k_L) & \ds\mathcal{H}_n (k_{L})\\
0&0
\end{bmatrix} \nonumber
\\\label{eqn:PTM}
&\ds\qquad\qquad\times \prod_{j=1}^{L}
\begin{bmatrix}
\ds h_n^{(1)} (k_j r_j)  & \ds -\frac{1}{\epsilon_j} \mathcal{H}_n (k_j r_j) \\
\ds -j_n(k_j r_j) &\ds \frac{1}{\epsilon_j} \mathcal{J}_n (k_j r_j)
\end{bmatrix}
\begin{bmatrix}
\ds\frac{1}{\epsilon_{j-1}} \mathcal{J}_n(k_{j-1}r_j) & \ds\frac{1}{\epsilon_{j-1}} \mathcal{H}_n (k_{j-1} r_j)\\
\ds j_n (k_{j-1} r_j)  & \ds h_n^{(1)} (k_{j-1} r_j)
\end{bmatrix} .
\end{align}
From the definition of $W_n^{TM}$ and \eqnref{00_P_TM}, we have
\begin{equation}\label{Wntm}W_n^{TE}=-\frac{in(n+1)}{k_0}\frac{b_0}{\tilde{b}_0}=-\frac{in(n+1)}{k_0}\frac{p^{TM}_{n,1}}{p^{TM}_{n,2}}.
\end{equation}

It should be emphasized that
$p_{n,2}^{TE} \neq 0$ and $p_{n,2}^{TM} \neq 0$. In fact, if
$p_{n,2}^{TE} = 0$, then \eqnref{00_P_TE} can be fulfilled with
$\tilde{a}_0=0$ and $a_0=1$. This means that there exists
$(\mu,\epsilon)$ on $\RR^3 \setminus \overline{D}$ such that the
following problem has a solution:
\begin{equation*}
 \ \left \{
 \begin{array}{ll}
\ds\nabla\times{\bE} = i\omega\mu{\bH} \quad &\mbox{in } \RR^3 \setminus \overline{D},\\
\ds\nabla\times{\bH} = -i\omega\epsilon{\bE} \quad &\mbox{in } \RR^3 \setminus \overline{D},\\
\ds(\Bx \times \bE)\bigr|_+ = 0 \quad &\mbox{on } \p D,\\
\ds \bE(\Bx) = \bE_{n,0}^{TE}(\Bx) \quad &\mbox{for } |\Bx| > 2.
 \end{array}
 \right .
\end{equation*}
Applying the following Green's theorem on $\Omega=\{\Bx\ \bigr| \ 1<|\Bx|<R\}$,
\begin{align*}
&\int_\Omega\bigr(\bE\cdot\Delta \mathbf{F}+\mbox{curl} \bE\cdot \mbox{curl} \mathbf{F}+\mbox{div} \bE\ \mbox{div} \mathbf{F}\bigr)d\Bx\\
&=\int_{\p \Omega}\bigr(\nu\times \bE\cdot\mbox{curl}\mathbf{F}+\nu\cdot\bE\ \mbox{div} \mathbf{F}\bigr)d\sigma(\Bx)
\end{align*}
with $\mathbf{F}=\overline{\bE_{n,0}^{TE}}(\Bx)$ and the PEC boundary condition on $\{\Bx|=1\}$,
we have
$$\int_{|\Bx|=R}(\nu\times\bE)\cdot\overline{\bH}d\sigma(\Bx)=ik_0\int_{\Omega}(|\bH|^2-|\bE|^2)d\Bx.$$
In particular, the left-hand side is real-valued. Hence,
\begin{align*}
\int_{|\Bx|=R}|\bH\times\nu-\bE|^2d\sigma(\Bx)
=&\int_{|\Bx|=R}\bigr(|\bH\times\nu|^2+|\bE|^2-2\Re((\nu\times\bE)\cdot\overline{\bH}\bigr)d\sigma(\Bx)\\
=&\int_{|\Bx|=R}\bigr(|\bH\times\nu|^2+|\bE|^2\bigr)d\sigma(\Bx).
\end{align*}
From the radiation condition, the left-hand side goes to zero as $R\rightarrow\infty$, and it contradicts the behavior of the hankel functions.
One can show that $p_{n,2}^{TM} \neq 0$ in a similar way.
\qed


To construct the S-vanishing structure at a fixed frequency $\omega$ we look for $(\mu,\epsilon)$ such that
$$
W_n^{TE} [\ep,\mu,\omega]=0, \ W_n^{TM} [\ep,\mu,\omega]=0, \quad n=1, \ldots, N
$$
for some $N$. More ambitiously we may look for a structure
$(\mu,\epsilon)$ for a fixed $\omega$ such that
$$ W^{TE}_n[\mu, \epsilon, \rho \omega] = 0,\quad  W^{TM}_n[\mu, \epsilon, \rho \omega] = 0 \quad
$$ for all $1 \leq n \leq N$ and $\rho \leq \rho_0$ for some $\rho_0$. Such a structure may not exist. So
instead we look for a structure such that
\beq\label{Sstructure}
W_n^{TE}[\mu, \epsilon, \rho \omega] = o(\rho^{2N+1}), \quad W_n^{TM}[\mu, \epsilon, \rho \omega] = o(\rho^{2N+1}),
\eeq
for all $1 \leq n \leq N $ and $\rho \leq \rho_0$ for some $\rho_0$. We call such a structure
a \textit{S-vanishing structure of order $N$ at low frequencies}.
In the following section, we expand the scattering coefficients for low frequencies and derive conditions for the magnetic permeability and the electric permittivity to be a S-vanishing structure.

Suppose that $(\mu, \epsilon) $ is an S-vanishing structure of order $N$ at low frequencies. Let the incident wave $\bE^i$ be given by a plane wave $e^{i \rho\mathbf{k}\cdot \Bx}\mathbf{c}$ with $|\mathbf{k}|=k_0$ and $\mathbf{k}\cdot \mathbf{c}=0$.
From \eqnref{far_field}, the corresponding scattering amplitude, $\mathbf{A}_{\infty}[\mu, \epsilon, \rho\omega](\bm{c},\hat{\bm{k}};\hat{\bm{x}})$, is given by \eqnref{aexpansion} with the following $\alpha_{n,m}$ and $\beta_{n,m}$.
\begin{equation*}
\begin{cases}
\ds\alpha_{n,m} = \frac{4 \pi i^n}{\sqrt{n(n+1)}}(\mathbf{V}_{n,m}(\hat{\mathbf{k}}) \cdot \mathbf{c} )W_{n}^{TE}[\mu, \epsilon, \rho\omega],\\
 \ds\beta_{n,m} =  -\frac{4 \pi i^n}{\sqrt{n(n+1)}}\frac{1}{i \omega \mu_0} (\mathbf{U}_{n,m}(\hat{\mathbf{k}}) \cdot \mathbf{c}) W_{n}^{TM}[\mu, \epsilon, \rho\omega].
\end{cases}
\end{equation*}
Applying \eqnref{W_decay_rho} and \eqnref{Sstructure}, we have
\beq \label{farfield1}
\ds \bm{A}_\infty[ \mu, \epsilon, \rho \omega] (\mathbf{c},\hat{\mathbf{k}};\hat{\Bx}) = o ( \rho^{2N+1} )
\eeq
uniformly in $(\hat{\mathbf{k}},\hat{\Bx})$ if $ \rho \leq \rho_0 $. Thus using such a structure the visibility of scattering amplitude is greatly reduced.

\subsection{Asymptotic expansion of the scattering coefficients}

The spherical Bessel functions of the first and second kinds have
the series expansions:
 $$j_n(t)=\sum_{l=0}^\infty \frac{(-1)^l t^{n+2l}}{2^l l!1\cdot 3\cdots(2n+2l+1)},$$
 and
 $$y_n(t) = -\frac{(2n)!}{2^n n!}\sum_{l=0}^\infty\frac{(-1)^l t^{2l-n-1}}{2^l l!(-2n+1)(-2n+3)\cdots(-2n+2l-1)}.$$
So, using the notation of double factorials, which is defined by
$$
n!!:= \left\{
\begin{array}{ll}
n \cdot (n-2) \ldots 3 \cdot 1 \quad & \mbox{if $n>0$ is odd}, \\
n \cdot (n-2) \ldots 4 \cdot 2 \quad & \mbox{if $n>0$ is even}, \\
1 \quad & \mbox{if $n=-1, 0$},
\end{array}
\right.
$$
we have \beq \label{jn_expansion}
j_n(t)=\frac{t^n}{(2n+1)!!}\bigr(1+o(t)\bigr)\quad\mbox{for
}t\ll1,\eeq and \beq \label{yn_expansion}
y_n(t)=-\bigr((2n-1)!!\bigr)t^{-n+1}\bigr(1+o(t)\bigr)\quad\mbox{for
}t\ll1.\eeq

We now compute $P_n^{TE}[\epsilon, \mu, t]$ for small $t$.
For $n\geq 1$,
 \begin{align*} P_{n}^{TE}[\ep,\mu,t]
&= (-i t)^L \left( \prod_{j=1}^{L} \mu_j^{\frac{3}{2}}
\ep_j^{\frac{1}{2}} r_j \right)
\begin{bmatrix}
\ds\frac{z_L^n}{(2n+1)!!}t^n+o(t^n)&\ds\frac{-iQ(n) }{z_L^{n+1}}t^{-n-1} \\
0&0
\end{bmatrix} \nonumber\\
&\times\prod_{j=1}^{L}\left(
\begin{bmatrix}
\ds \frac{i Q(n) n }{\mu_j (z_j r_j)^{n+1}} t^{-n-1} +o(t^{-n-1}) &\ds \frac{i Q(n)}{(z_j r_j)^{n+1}}t^{-n-1} +o(t^{-n-1}) \\
\nm
\ds\frac{-(n+1)(z_j r_j)^n}{\mu_j (2n+1)!!} t^{n} +o(t^{n})& \ds \frac{(z_j r_j)^n}{(2n+1)!!}t^n +o(t^n)
\end{bmatrix} \right.\nonumber\\
&
\qquad\quad\left.\begin{bmatrix} \ds\frac{(z_{j-1}r_j)^n}{(2n+1)!!}t^n +o(t^n)& \ds\frac{-i Q(n)}{(z_{j-1}r_j)^{n+1}}t^{-n-1}+o(t^{-n-1})\\ \nm
\ds\frac{(n+1)(z_{j-1} r_j)^n}{\mu_{j-1}(2n+1)!!}t^n +o(t^n) & \ds
\frac{i Q(n) n }{\mu_{j-1} (z_{j-1} r_j)^{n+1}}t^{-n-1} +o(t^{-n-1})
\end{bmatrix} \right),\nonumber
\end{align*}
where $z_j=\sqrt{ \ep_j \mu_j }$ and $Q(n)=(2n-1)!!$.
We then have
{\small \small
\begin{align*}
&P_{n}^{TE}[\ep,\mu,t] =\begin{bmatrix}
\ds\frac{z_L^n}{(2n+1)!!}t^n+o(t^n)&\ds\frac{-iQ(n) }{z_L^{n+1}}t^{-n-1} +o(t^{-n-1})\\
0&0
\end{bmatrix}\times\\
&
\prod_{j=1}^{L} \nonumber
\begin{bmatrix}
\ds \frac{Q(n) z_{j-1}^n }{(2n+1)!!z_j^n}\bigr( n + \frac{(n+1)\mu_j}{\mu_{j-1}}\bigr)\bigr(1 +o(1) \bigr)&
\ds (-i)\frac{(Q(n))^2 n }{z_j^n z_{j-1}^{n+1} r_j^{2n+1}} \bigr( 1-\frac{\mu_j}{\mu_{j-1}}\bigr){t^{-2n-1}}{\bigr(1 +o(1) \bigr)}  \\
\nm
\ds i\frac{z_{j-1}^n z_j^{n+1} r_j^{2n+1}(n+1) }{((2n+1)!!)^2 }\bigr(1-\frac{\mu_j}{\mu_{j-1}}\bigr)t^{2n+1} \bigr(1 +o(1) \bigr) & \ds \frac{Q(n) z_j^{n+1}}{(2n+1)!! z_{j-1}^{n+1}}\bigr( n+1+\frac{n\mu_j}{\mu_{j-1}} \bigr) \bigr(1 +o(1) \bigr)
\end{bmatrix}.
\end{align*}}

Similarly, for the transverse magnetic case, we have {\small \small \begin{align*}
&P_{n}^{TM}[\ep,\mu,t] =
\begin{bmatrix}
\ds \frac{(n+1)z_L^n}{(2n+1)!!}t^n +o(t^n)&\ds \frac{-i n Q(n)}{z_L^{n+1}} t^{-n-1} +o(t^{-n-1})\\
\ds0&\ds0
\end{bmatrix} \times \nonumber \\
&\prod_{j=1}^{L}
\begin{bmatrix}
\ds \frac{Q(n) z_{j-1}^n }{(2n+1)!!z_j^n}\left( (n + \frac{\ep_j}{\ep_{j-1}}(n+1)\right)\bigr(1 +o(1) \bigr) & \ds (-i)\frac{(Q(n))^2 n }{z_j^n z_{j-1}^{n+1} r_j^{2n+1}} \left( 1-\frac{\ep_j}{\ep_{j-1}}\right)t^{-2n-1}\bigr(1 +o(1) \bigr)  \\
\nm
\ds i\frac{ z_{j-1}^n z_j^{n+1} r_j^{2n+1}(n+1) }{((2n+1)!!)^2 }\left(1-\frac{\ep_j}{\ep_{j-1}}\right)t^{2n+1} \bigr(1 +o(1) \bigr) & \ds \frac{Q(n) z_j^{n+1}}{(2n+1)!! z_{j-1}^{n+1}}\left( n+1+\frac{\ep_j}{\ep_{j-1}} n\right) \bigr(1 +o(1) \bigr)
\end{bmatrix}.
\end{align*}}

Using the behavior of spherical bessel functions for small arguments,
we see that $p_{n,1}^{TE}$ and $p_{n,2}^{TE}$ admit the following expansions:
\beq
p_{n,1}^{TE}[\mu,\ep,t] = t^n \left( \sum_{l=0}^{N-n} f_{n,l}^{TE}(\mu,\ep) t^{2 l} + o(t^{2N-2n}) \right)
\eeq
and
\beq
p_{n,2}^{TE}[\mu,\ep,t] = t^{-n-1} \left(  \sum_{l=0}^{N-n} g_{n,l}^{TE}(\mu,\ep) t^{2 l} + o(t^{2N-2n})  \right).
\eeq
Similarly, $p_{n,1}^{TM}$ and $p_{n,2}^{TM}$  have the following expansion:
\beq
p_{n,1}^{TM}[\mu,\ep,t] = t^n \left( \sum_{l=0}^{N-n} f_{n,l}^{TM}(\mu,\ep) t^{2 l} + o(t^{2N-2n}) \right)
\eeq
and
\beq
p_{n,2}^{TM}[\mu,\ep,t] = t^{-n-1} \left(  \sum_{l=0}^{N-n} g_{n,l}^{TM}(\mu,\ep) t^{2 l} + o(t^{2N-2n})  \right).
\eeq
for $t=\rho \omega$ and some functions $ f_{n,l}^{TE}, g_{n,l}^{TE},  f_{n,l}^{TM},$ and $g_{n,l}^{TM}$ independent of $t$.
\begin{lemma}\label{lemma:leading}
For any pair of $(\mu, \epsilon)$, we have
\beq \label{g_n0_TE}
g_{n,0}^{TE}(\mu,\ep) \neq 0
\eeq
and
\beq \label{g_n0_TM}
 g_{n,0}^{TM}(\mu,\ep)  \neq 0.
\eeq
\end{lemma}
\begin{proof}
Assume that there exists a pair of $(\mu, \epsilon)$ such that $g_{n,0}^{TE}(\mu,\ep) = 0$.
Since $p_{n,2}^{TE}[\mu,\ep,\rho\omega]=o(\rho^{-n-1})$, the solution given by \eqnref{E_n_TE} with $a_0=1$ and $\tilde{a}_0=0$ satisfies
\begin{equation*}
 \ \left \{
 \begin{array}{ll}
\ds \nabla \times\left( \frac{1}{\mu}
\nabla \times \bE\right) - \rho^2 \omega^2 \epsilon \bE = 0 \quad &\mbox{in } \RR^3 \setminus \overline{D},\\
\ds \nabla \cdot \bE =0  \quad &\mbox{in } \RR^3 \setminus \overline{D},\\
\ds(\bm{\nu} \times \bE)\bigr|_+ = o(\rho^{-(n+1)}) \quad &\mbox{on } \p D,\\
\ds\bE(\Bx) = h_n^{(1)} ( \rho k_0 |\Bx| ) \bm{V}_{n,0}(\hat{\Bx})
\quad &\mbox{for } |\Bx| > 2.
 \end{array}
 \right .
 \end{equation*}
Let $\bm{V}(\Bx)=\lim_{\rho \rightarrow 0} \rho^{n+1} \bE(\Bx)$.
Using \eqnref{yn_expansion}
we know that the limit $\bm{V}$ satisfies
\begin{equation*}
 \ \left \{
 \begin{array}{ll}
\ds \nabla \times\left( \frac{1}{\mu} \nabla \times \bm{V} \right) = 0 \quad &\mbox{in }\RR^3 \setminus \overline{D},\\
\ds \nabla \cdot \bm{V} =0 \quad &\mbox{in } \RR^3 \setminus \overline{D}, \\
\ds (\bm{\nu} \times \bm{V})\bigr|_+ = 0 \quad &\mbox{on } \p D,\\
\ds \bm{V}(\Bx) = -\bigr((2n-1)!!)\bm{V}_{n,0}(\hat{\Bx})
\quad &\mbox{for } |\Bx| > 2.
 \end{array}
 \right .
 \end{equation*}
Since $\bm{V}_{n,0}(\hat{\Bx})=O(|x|^{-1})$, we get $\bm{V}(\Bx)=0$ by Green's formula, which is a contradiction. Thus $g_{n,0}^{TE}(\mu,\ep) \neq 0$.  In a similar way, \eqnref{g_n0_TM} can be proved.
\end{proof}

From Lemma \ref{lemma:leading}, we have the following theorem.
\begin{prop} \label{scatt_coeff_t}
We have
$$ W_{n}^{TE}[\mu,\ep,t] = t^{2n+1} \sum_{l=0}^{N-n}
W_{n,l}^{TE}[\mu,\ep] t^{2l} + o(t^{2N+1}) ,$$
 and
 $$
W_{n}^{TM}[\mu,\ep,t] = t^{2n+1} \sum_{l=0}^{N-n}
W_{n,l}^{TM}[\mu,\ep] t^{2l} + o(t^{2N+1}), $$
where $t=\rho\omega$ and the coefficients $W_{n,l}^{TE}[\mu,\ep]$ and
$W_{n,l}^{TM}[\mu,\ep]$ are independent of $t$.
\end{prop}

Hence, if we have $(\mu,\epsilon)$ such that
\beq\label{eqn:Scondition} W_{n,l}^{TE}[\mu, \epsilon] = W_{n,l}^{TM}[\mu,\ep]=0,\quad\mbox{for all}~1 \leq n \leq N,\ 0\leq l\leq (N-n),
\eeq
$(\mu,\epsilon)$ satisfies \eqnref{Sstructure}, in other words,
it is a \textit{S-vanishing structure of order $N$ at low frequencies}.
It is quite challenging to construct $(\mu,\epsilon)$ analytically satisfying \eqnref{eqn:Scondition}. In the next section we get some numerical examples of such structures.

\subsection{Numerical examples}
In this section we provide numerical examples of S-vanishing structures of order $N$ at low frequencies based on \eqnref{eqn:Scondition}.
To do this, the gradient descent method for the suitable energy functional is used, as used in \cite{AKLL1} and \cite{AKLL2} to compute the enhanced near-cloaking structures for the conductivity problem and the Helmholtz problem.
As in \cite{AKLL2}, we symbolically compute the scattering coefficients. In the place of spherical Bessel functions and spherical Hankel functions, we put its low frequency asymptotic expansions in \eqnref{eqn:PTE} and \eqnref{eqn:PTM}, and symbolically compute  $W_n^{TE}$ and $W_n^{TM}$ to have $W_{n,l}^{TE}[\mu, \epsilon]$ and $W_{n,l}^{TM}[\mu,\ep]$.

The following example is a S-vanishing structure of order $N=2$ made of 6 multilayers. The radii of the concentric disks are
 $r_j = 2-\frac{j-1}{6}$ for $j=1,\dots,7$. From Proposition \ref{scatt_coeff_t},
the nonzero leading terms of $W_n^{TE}[ \mu, \epsilon, t]$ and $W_n^{TM}[ \mu, \epsilon, t]$ up to $t^5$ are
\begin{itemize}
\item $[t^3,t^5]$ terms in  $W_1^{TE}[ \mu, \epsilon, t]$, \textit{i.e.},  $W_{1,0}^{TE}, W_{1,1}^{TE}$,
\item $[t^3,t^5]$ terms in  $W_1^{TM}[ \mu, \epsilon, t]$, \textit{i.e.},  $ W_{1,0}^{TM}, W_{1,1}^{TM}$,
\item $[t^5]$ term in  $W_2^{TE}[ \mu, \epsilon, t]$, \textit{i.e.},  $W_{2,0}^{TE}$,
\item $[t^5]$ term in  $W_2^{TM}[ \mu, \epsilon, t]$, \textit{i.e.},  $W_{2,0}^{TM}$.
\end{itemize}
Consider the mapping
 \beq\label{Msigma}
 \ds({\boldsymbol\mu},\bm{\ep})\longrightarrow (W_{1,0}^{TE}, W_{1,1}^{TE},  W_{1,0}^{TM}, W_{1,1}^{TM},  W_{2,0}^{TE}, W_{2,0}^{TM}),
 \eeq
where, $\bm{\mu}=(\mu_1,\dots,\mu_6)$ and $\bm{\ep}=(\ep_1,\dots,\ep_6)$. We look for $(\bm{\mu}, \bm{\ep})$ which has the right-hand side of \eqnref{Msigma} as small as possible. Since \eqnref{Msigma} is a nonlinear equation, we solve it iteratively. Initially,  we wet $\bm{\mu}=\bm{\mu}^{(0)}$ and $\bm{\ep}=\bm{\ep}^{(0)}$.
We iteratively modify $(\bm{\mu}^{(i)},\bm{\ep}^{(i)})$
\beq\label{iter}[\boldsymbol\mu^{(i+1)}\ \bm{\ep}^{(i+1)}]^T = \boldsymbol[\boldsymbol\mu^{(i)}\ \bm{\ep}^{(i)}]^T - A_i^\dag \textbf{b}^{(i)},
\eeq
where $A_i^\dag$ is the pseudoinverse of
$$
A_i:=\frac{\partial (W_{1,0}^{TE}, W_{1,1}^{TE}, \dots, W_{2,0}^{TM})}{\partial( \bm{\mu},\bm{\ep})}\Bigr|_{( \bm{\mu},\bm{\ep}) =  ( \bm{\mu}^{(i)},\bm{\ep}^{(i)})},$$
and
$$\textbf{b}^{(i)} = \left.\begin{bmatrix}
W_{1,0}^{TE} \\ W_{1,1}^{TE}\\\vdots\\W_{2,0}^{TM} \end{bmatrix}\right|_{( \bm{\mu},\bm{\ep}) =  ( \bm{\mu}^{(i)},\bm{\ep}^{(i)})}.$$

\noindent{\bf Example 1}. Figure \ref{figure:1} and Figure \ref{figure:2} show
computational results of 6-layers S-vanishing structure of order $N=2$. We set $\bm{r}=(2, \frac{11}{6},\dots, \frac{7}{6})$, $\bm{\mu}^{(0)}=(3, 6, 3, 6, 3, 6)$ and $\bm{\ep}^{(0)}=(3, 6, 3, 6, 3, 6)$ and modify them following \eqnref{iter} with the constraints that $\mu$ and $\ep$ belongs to the interval between 0.1 and 10. The obtained material parameters are $\bm{\mu}=(0.1000, 1.1113, 0.2977, 2.0436, 0.1000, 1.8260)$ and $\bm{\ep}=(0.4356, 1.1461, 0.2899, 1.8199, 0.1000, 3.1233)$, respectively. Differently from the no-layer structure with PEC condition at $|\Bx|=1$, the obtained multilayer structure has the nearly zero coefficients of $W_n^{TE}[ \mu, \epsilon, t]$ and $W_n^{TM}[ \mu, \epsilon, t]$ up to $t^5$.

\begin{figure}[h!]
\begin{center}
\hskip.2cm
\epsfig{figure=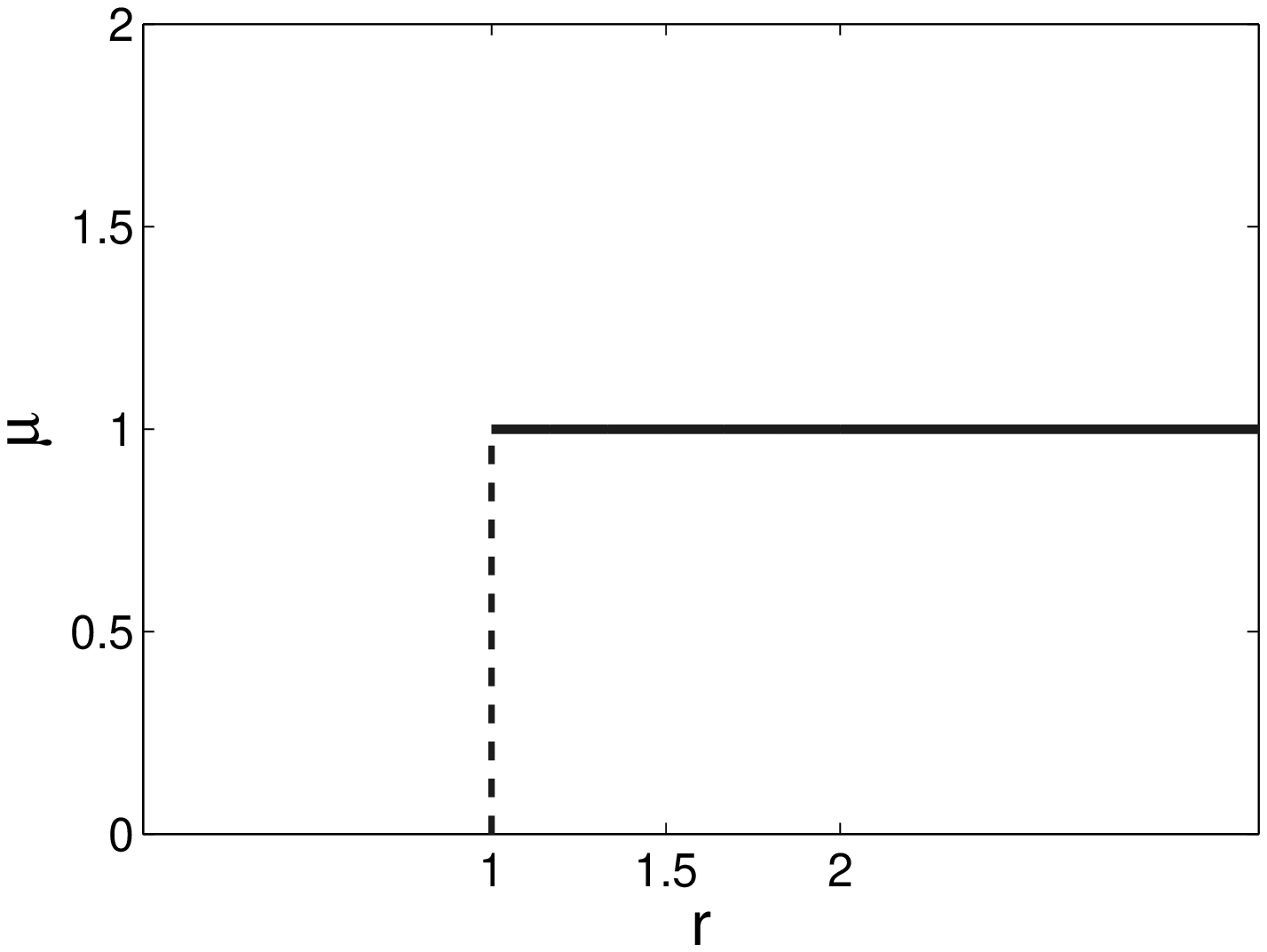, height=3cm}\hskip .2cm
\epsfig{figure=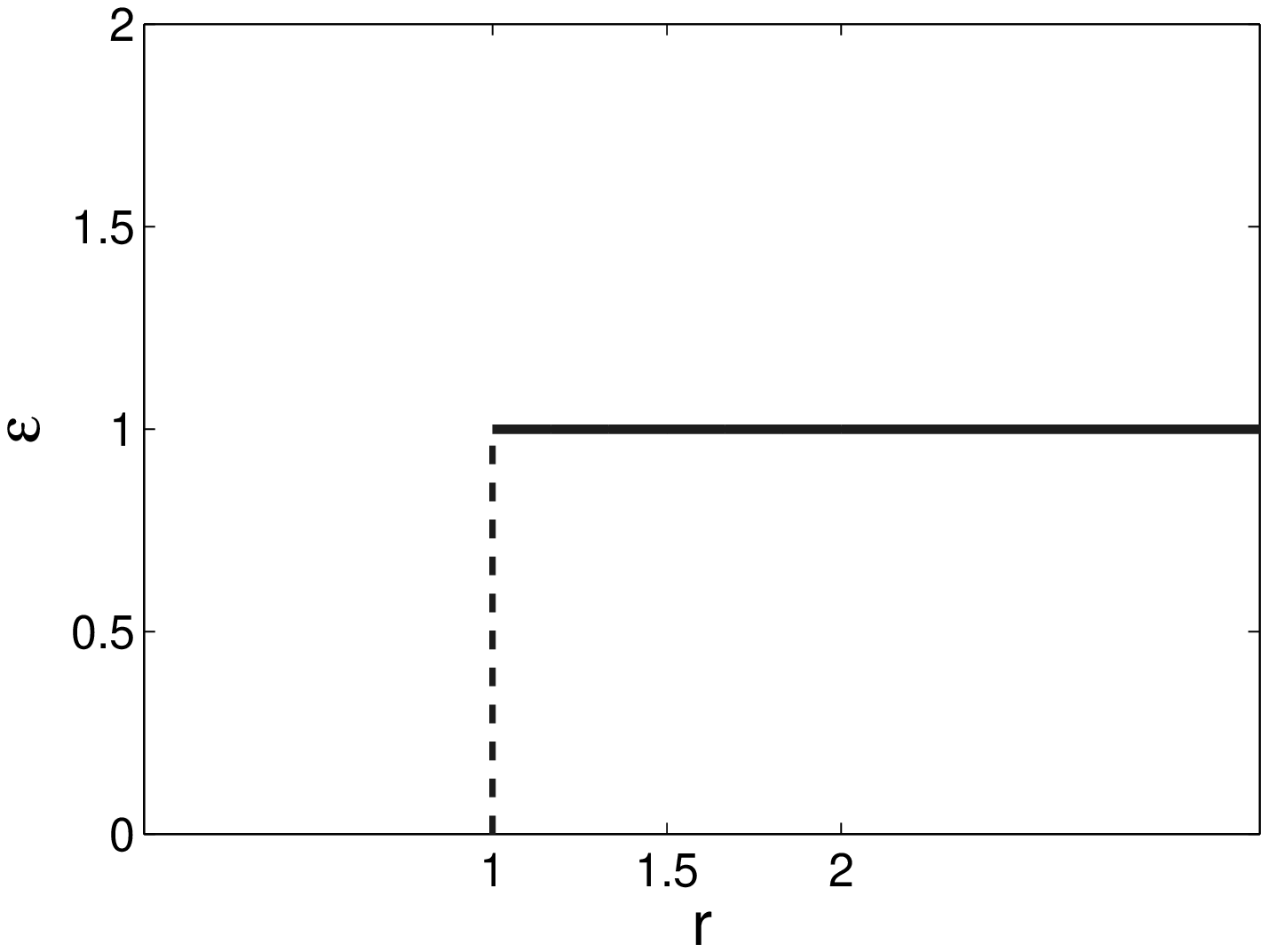, height=3cm}\hskip .3cm
\epsfig{figure=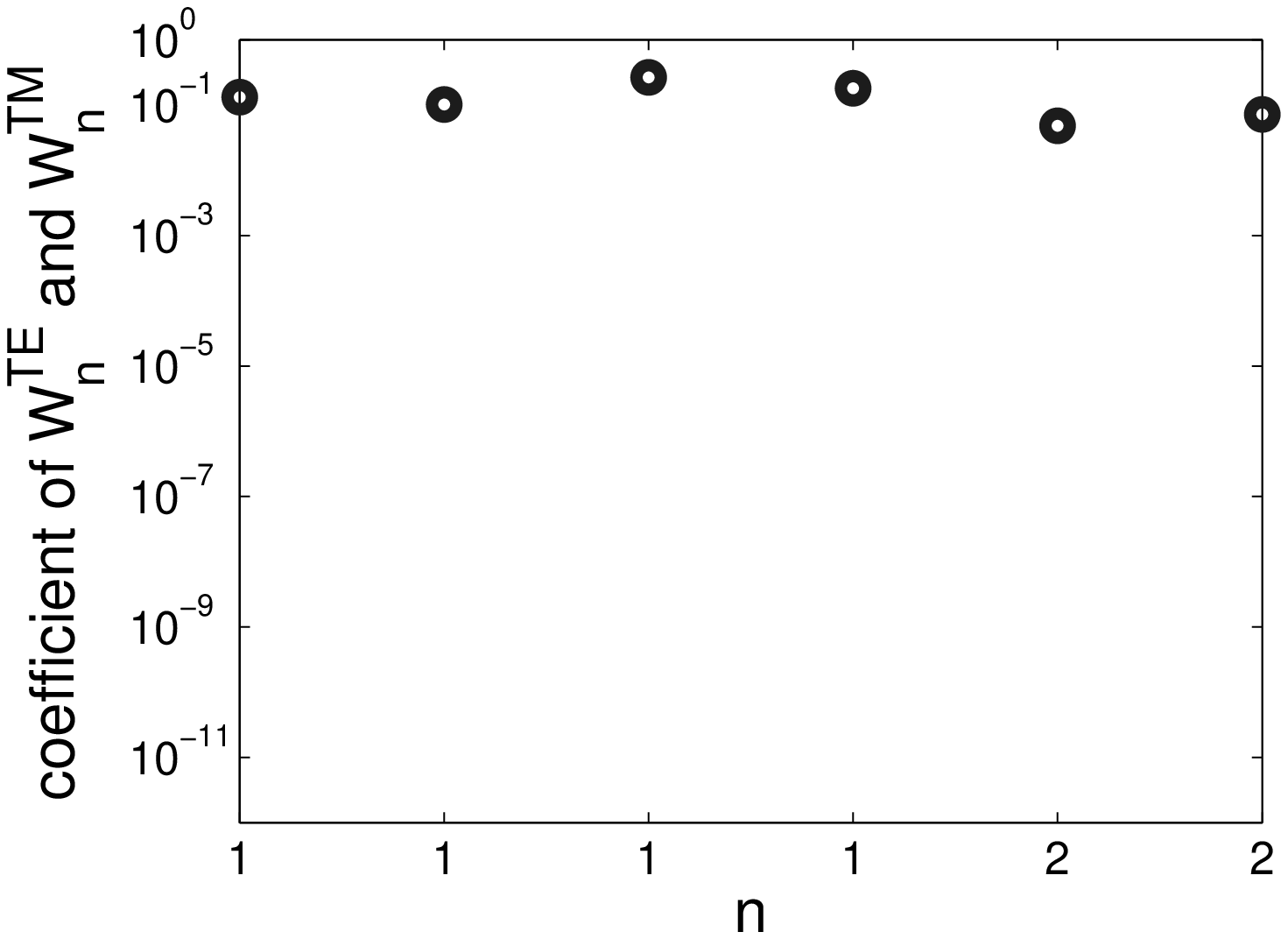, height=3cm}\\[.2cm]
\epsfig{figure=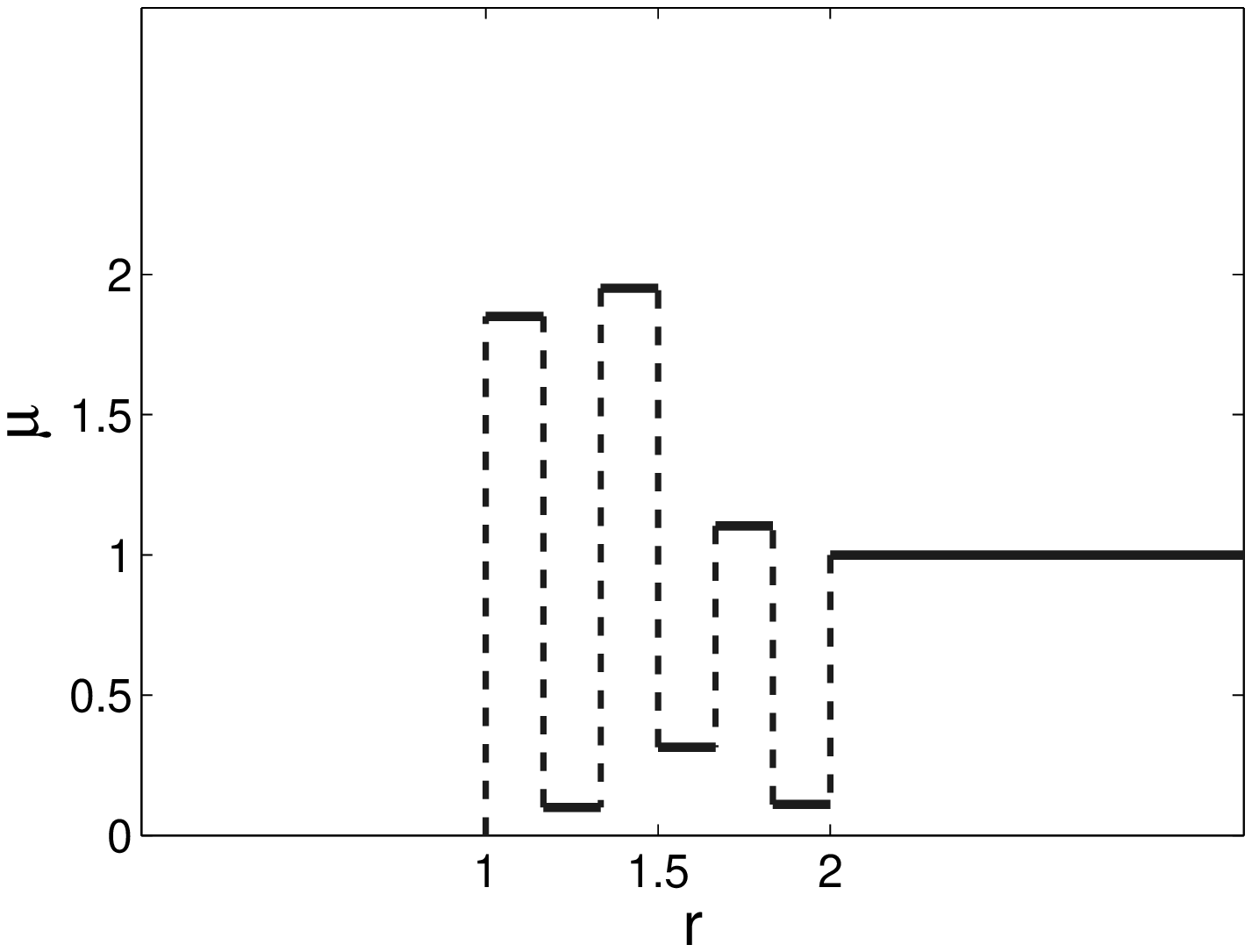, height=3cm}\hskip .2cm
\epsfig{figure=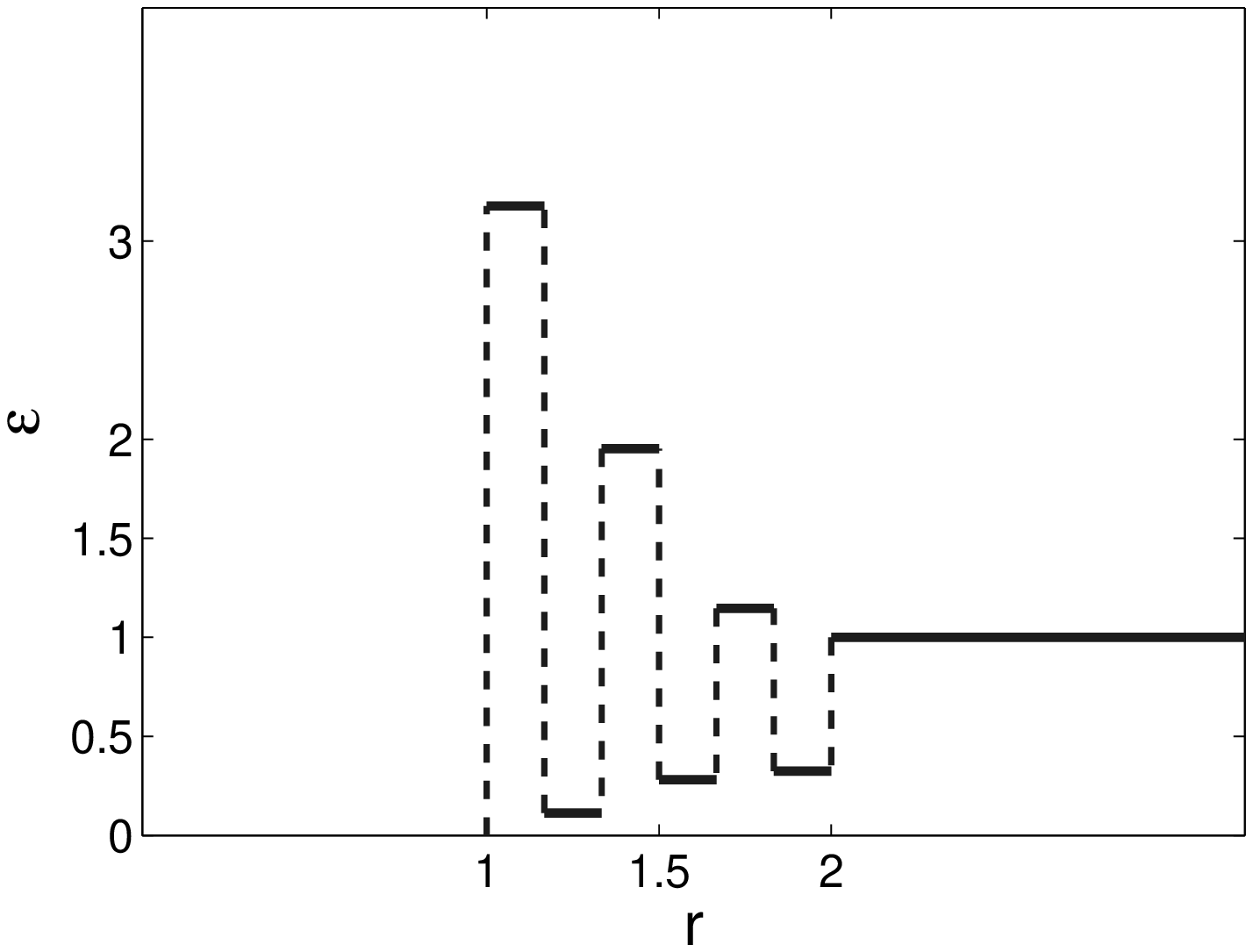, height=3cm}\hskip.3cm
\epsfig{figure=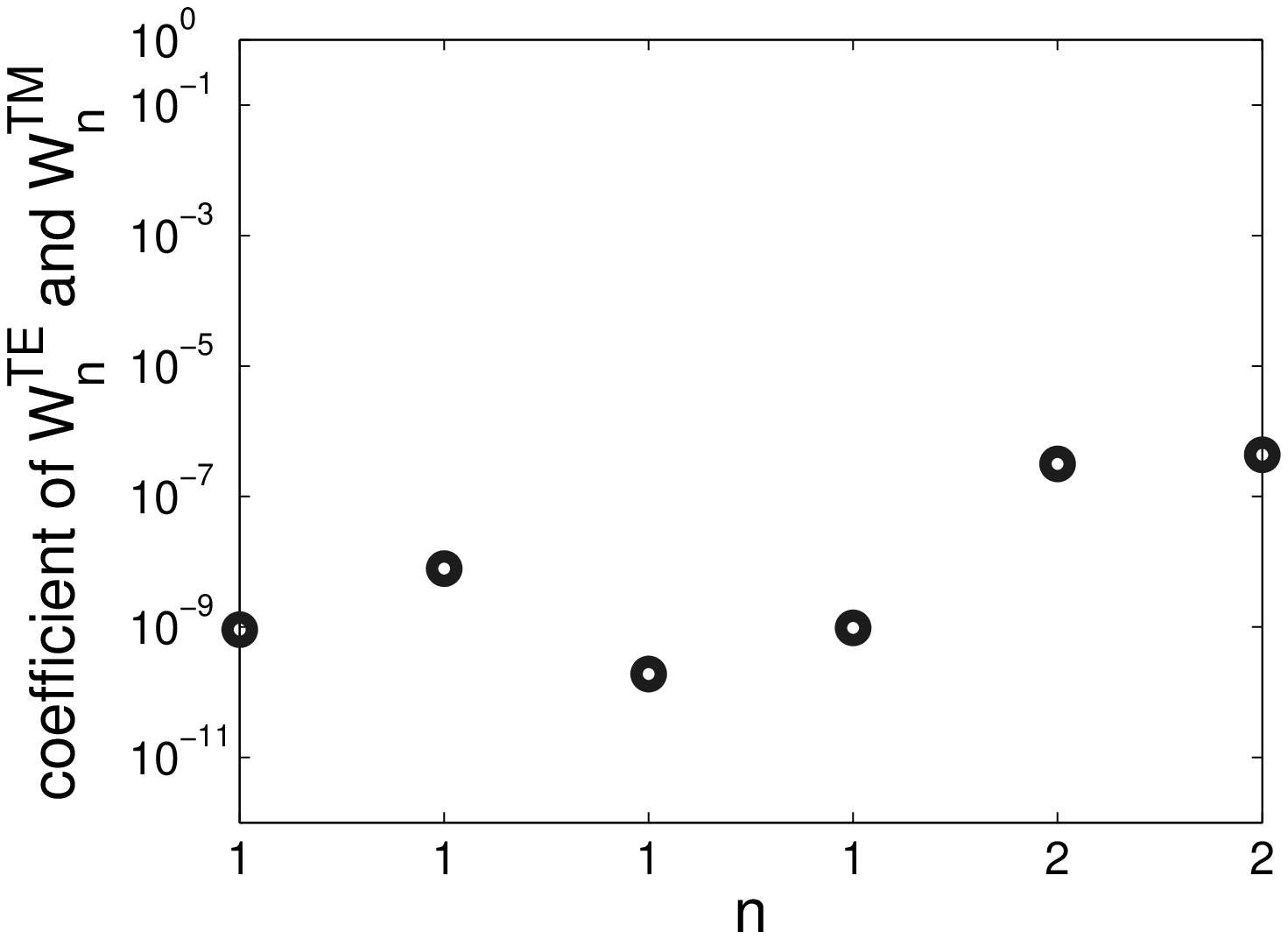, height=3cm}
\end{center}
\caption{This figure shows the graph of the material parameters and the corresponding coefficients in $W_n^{TE}[ \mu, \epsilon, t]$ and $W_n^{TM}[ \mu, \epsilon, t]$ up to $t^5$. The first row is of the no-layer case, and the second row is of 6-layers S-vanishing structure of order $N=2$ which is explained in Example 1. In the third column, the $y$-axis shows  $(W_{1,0}^{TE}, W_{1,1}^{TE},  W_{1,0}^{TM}, W_{1,1}^{TM},  W_{2,0}^{TE}, W_{2,0}^{TM})$ from the left to the right.}\label{figure:1}
\end{figure}

\begin{figure}[h!]
\begin{center}
\epsfig{figure=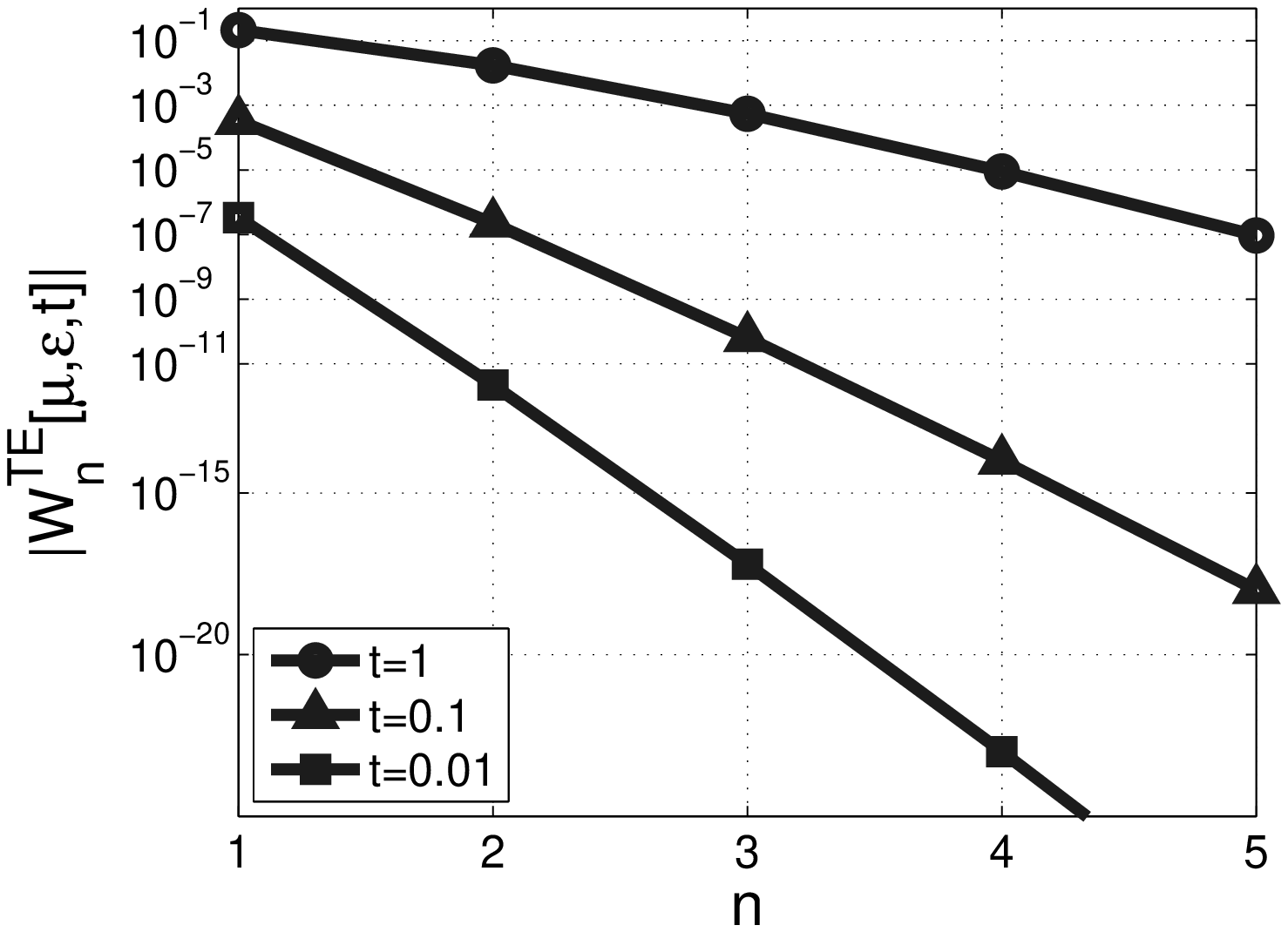,width=5cm}\hskip
.3cm \epsfig{figure=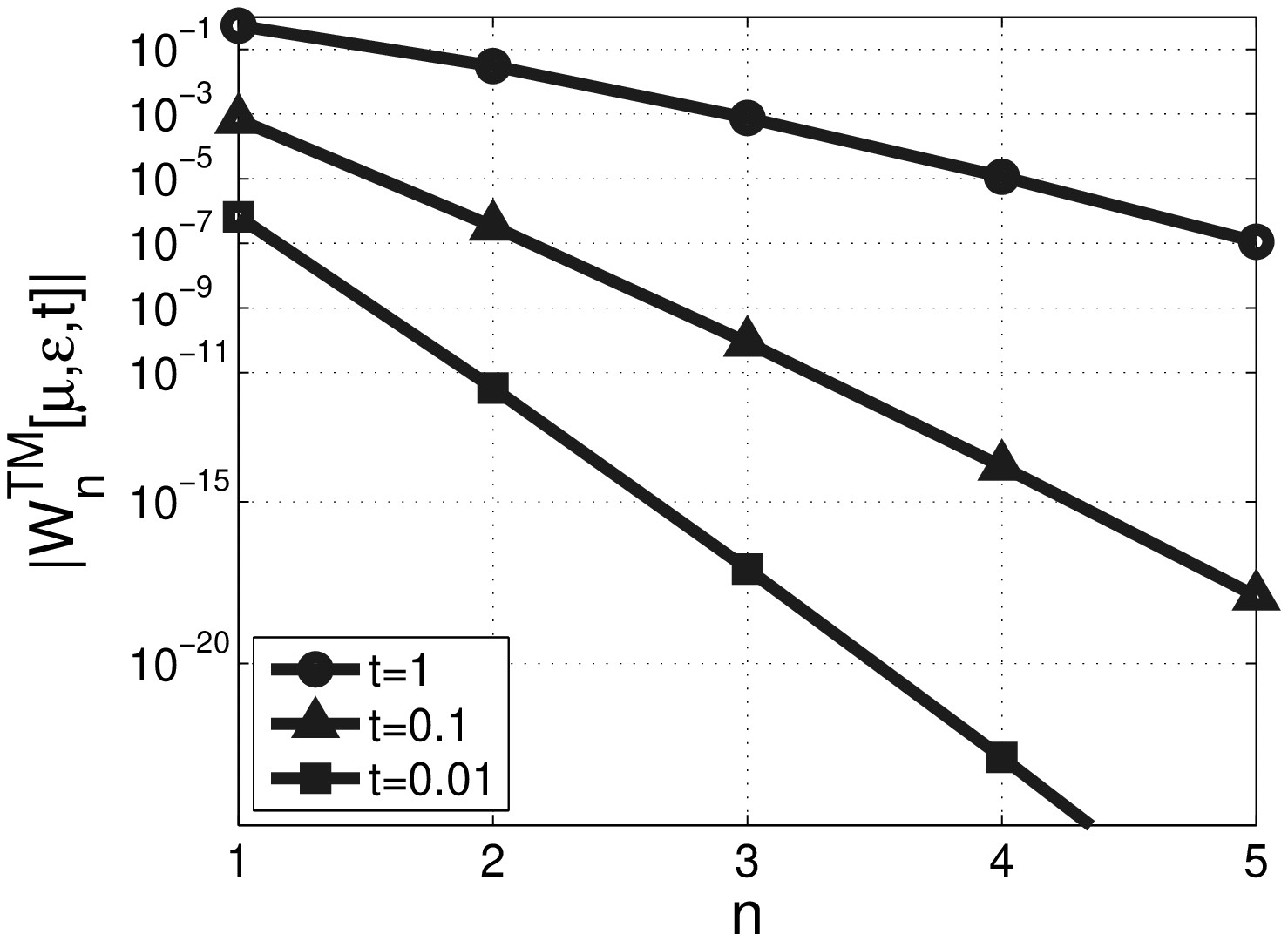,
width=5cm}\\[.2cm]
\epsfig{figure=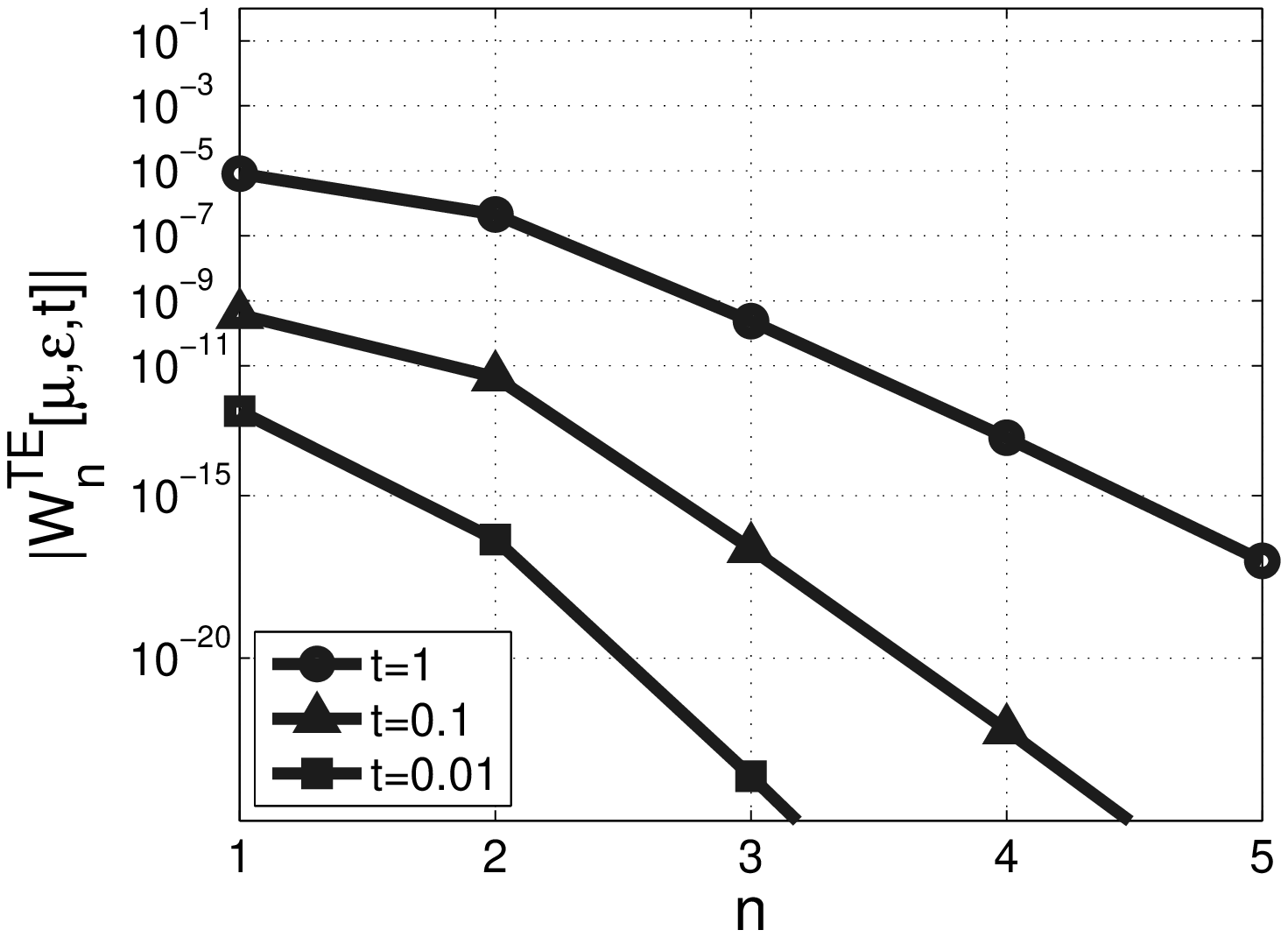,
width=5cm}\hskip .3cm
\epsfig{figure=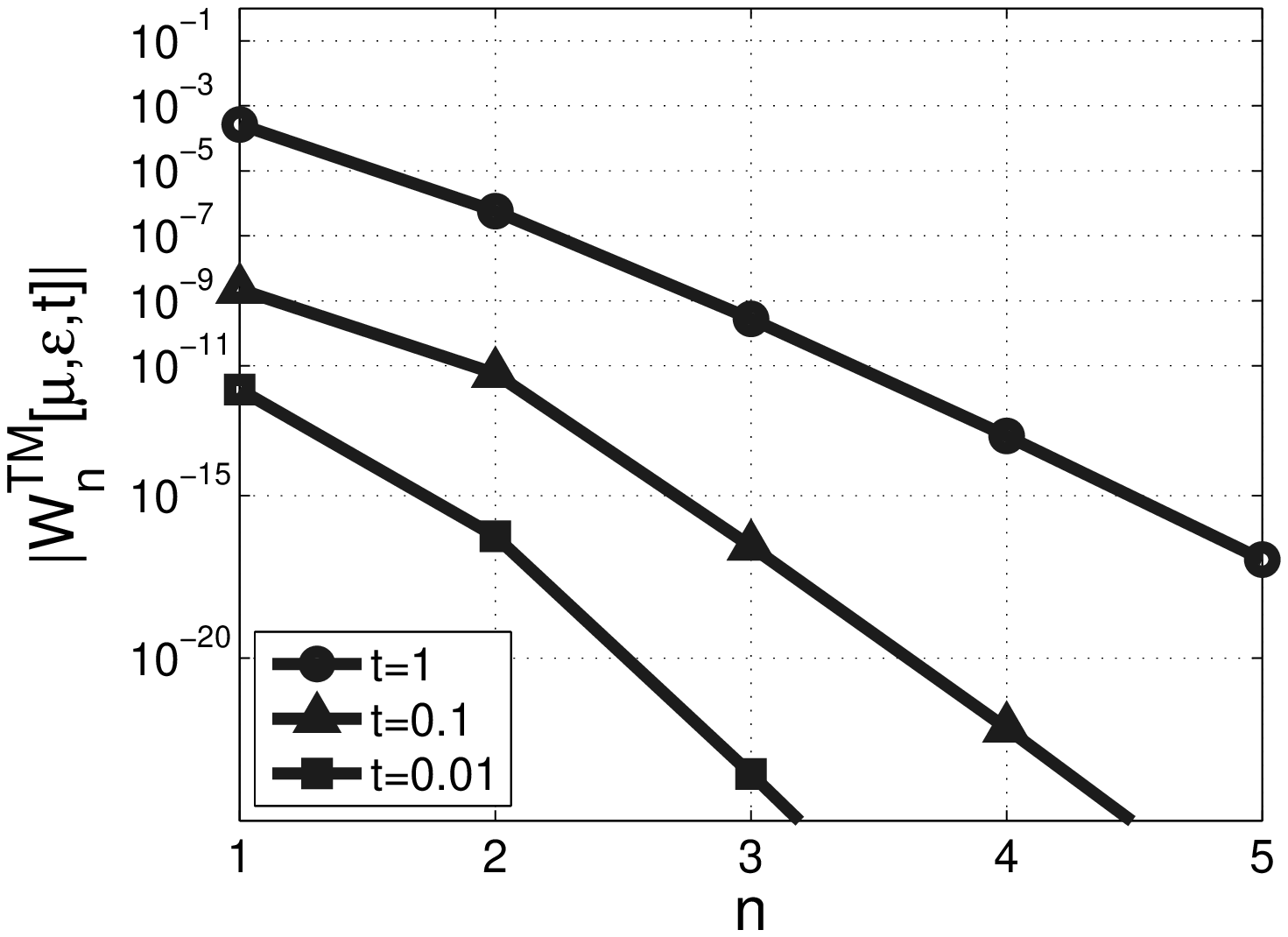,
width=5cm}
\end{center}
\caption{This figure shows the graph of $W_n^{TE}[ \mu, \epsilon, t]$ and $W_n^{TM}[ \mu, \epsilon, t]$ for various values of $t$.
The first row is of the no-layer case, and the second row if of 6-layers S-vanishing structure of order $N=2$ which is explained in Example 1. The values of $W_n^{TE}$ and $W_n^{TM}$ are much smaller in the S-vanishing structure than in the no-layer structure.}\label{figure:2}
\end{figure}

 \section{Enhancement of near cloaking}
We make a cloaking structure based on the following lemma.
\begin{lemma} \label{change_of_var}
Let $F$ be a diffeomorphism of $\RR^3 $onto $\RR^3$ such that
$F(\Bx)$ is identity for $|\Bx|$ large enough. If $(\bE,\bH)$ is a
solution to \beq
 \ \left \{
 \begin{array}{ll}
\nabla\times{\bE} = i\omega\mu{\bH} \quad &\mbox{in } \RR^3,\\
\nm
\nabla\times{\bH} = -i\omega\epsilon{\bE} \quad &\mbox{in } \RR^3,\\
\nm (\bE - {\bE}^i,\bH - {\bH}^i) \mbox{ is radiating,}
 \end{array}
 \right .
 \eeq
then $(\widetilde{\bE},\widetilde{\bH}) $ defined by
$(\widetilde{\bE}(\By),\widetilde{\bH}(\By))=\bigr(\bE(F^{-1}(\By)),\bH(F^{-1}(\By))\bigr)$
satisfies
\begin{equation*}
 \ \left \{
 \begin{array}{ll}
\ds\nabla\times\widetilde{\bE} = i\omega(F_*\mu)\widetilde{\bH} \quad &\mbox{in } \RR^3,\\
\ds\nabla\times\widetilde{\bH} = -i\omega(F_*\epsilon)\widetilde{\bE} \quad &\mbox{in } \RR^3,\\
\ds  (\widetilde{\bE} \widetilde{\bE}^i,\widetilde{\bH} -
\widetilde{\bH}^i) \mbox{ is radiating},
 \end{array}
 \right .
 \end{equation*}
where
$(\widetilde{\bE}^i(\By),\widetilde{\bH}^i(\By))=\bigr(\bE^i(F^{-1}(\By)),\bH^i(F^{-1}(\By))\bigr)$,
$$ (F_* \mu) (\By) = \frac{DF(\Bx) \mu(\Bx) DF^T(\Bx)}{\det
(DF(\Bx))}, \quad \mbox{and} \quad (F_* \epsilon) (\By) =
\frac{DF(\Bx) \epsilon(\Bx) DF^T(\Bx)}{\det (DF(\Bx))}, $$ with
$\Bx = F^{-1}(\By)$ and $DF$ is the Jacobian matrix of $F$.

Hence,
$$\bm{A}[\mu,\epsilon,\omega]=\bm{A}[F_*\mu,F_*\epsilon,\omega].$$
\end{lemma}

To compute the scattering amplitude which corresponds to the
material parameters before the transformation, we consider the
following scaling function, for small parameter $\rho$,
$${\Psi}_{\frac{1}{\rho}} (\Bx) = \frac{1}{\rho} \Bx, \quad
\Bx \in \RR^3.$$  Then we have the following relation between the
scattering amplitudes which correspond to two sets of differently
scaled material parameters and frequency:
 \beq \label{farfield3}\ds \bm{A}_{\infty}
\left[ \mu \circ {\Psi}_{\frac{1}{\rho}}, \epsilon \circ
{\Psi}_{\frac{1}{\rho}} , \omega \right] = \bm{A}_{\infty} [ \mu,
\epsilon, \rho \omega]. \eeq

To see this, consider $(\bE,\bH)$ which satisfies
\begin{equation*}
 \ \left \{
 \begin{array}{ll}
\ds\left( \nabla\times{\bE}\right) (\Bx) = i\omega \bigr( \mu \circ {\Psi}_{\frac{1}{\rho}}\bigr)(\Bx) {\bH}(\Bx) \quad &\mbox{for } \Bx \in \RR^3 \setminus \overline{B_{\rho}},\\
\ds\left( \nabla\times{\bH}\right) (\Bx) = -i\omega\bigr( \epsilon \circ {\Psi}_{\frac{1}{\rho}}\bigr)(\Bx){\bE}(\Bx) \quad & \mbox{for } \Bx \in \RR^3 \setminus \overline{B_{\rho}},\\
\ds\hat{\Bx} \times \bE(\Bx) = 0 \quad& \mbox{on } \p B_{\rho},\\
(\bE-\bE^i, \bH- \bH^i) \mbox{ is radiating,}&
 \end{array}
 \right .
 \end{equation*}
with the incident wave $\bE^i(\Bx) =  e^{i \mathbf{k} \cdot \Bx}\hat{\mathbf{c}}$
and $\bH^i= \frac{1}{i\omega \mu_0}\nabla\times{\bE}^i$ with
$\mathbf{k} \cdot \hat{\mathbf{c}} = 0$ and $|\mathbf{k}|=k_0$.
Here $B_{\rho}$ is the ball of radius $\rho$ centered at the origin.
Set $\By = \frac{1}{\rho} \Bx$ and define
$$ \bigr(
\widetilde{\bE}(\By), \widetilde{\bH} (\By)\bigr):= \Bigr( \bigr(
\bE \circ {\Psi}_{\frac{1}{\rho}}^{-1} \bigr) (\By) , \bigr( \bH
\circ {\Psi}_{\frac{1}{\rho}}^{-1} \bigr) (\By) \Bigr) =\Bigr(
\bigr( \bE \circ {\Psi}_{\rho} \bigr) (\By) , \bigr( \bH \circ
{\Psi}_{\rho} \bigr) (\By) \Bigr) $$
 and
 $$ \bigr(
\widetilde{\bE}^i(\By), \widetilde{\bH}^i (\By)\bigr):=\Bigr(
\bigr( \bE^i \circ {\Psi}_{\rho} \bigr) (\By) , \bigr( \bH^i
\circ {\Psi}_{\rho} \bigr) (\By) \Bigr). $$
 Then, we have
 \begin{equation*}
 \ \left \{
 \begin{array}{ll}
\ds\left( \nabla_\By\times\widetilde{\bE}\right) (\By) = i\rho\omega  \mu (\By) \widetilde{\bH}(\By) \quad &\mbox{for } \By \in \RR^3\setminus\overline{B_1}\\
\ds\left( \nabla_\By\times\widetilde{\bH}\right) (\By) = -i\rho\omega \epsilon (\By)\widetilde{\bE}(\By) \quad& \mbox{for } \By \in \RR^3\setminus\overline{B_1},\\
\ds\hat{\By} \times \widetilde{\bE}(\By) = 0 \quad &\mbox{on } \p B_1,\\
\ds(\widetilde{\bE}-\widetilde{\bE}^i,\widetilde{\bH}-\widetilde{\bH}^i) \mbox{ is radiating}
 \end{array}
 \right .
 \end{equation*}
 Remind that the scattered wave can be represented using the scattering amplitude as follows:
$$
(\bE-\bE^i) (\Bx) \sim \frac{e^{i k_0 |\Bx|}}{k_0 |\Bx|}
\bm{A}_{\infty} \left[ \mu \circ {\Psi}_{\frac{1}{\rho}}, \epsilon
\circ {\Psi}_{\frac{1}{\rho}} , \omega
\right](\mathbf{c},\hat{\mathbf{k}};\hat{\Bx}) \quad \mbox{as }
|\Bx| \rightarrow \infty,
$$
and
$$
(\widetilde{\bE}-\widetilde{\bE}^i) (\By) \sim \frac{e^{i k_0 \rho
|\By|}}{k_0 \rho |\By|} \bm{A}_{\infty} \left[ \mu , \epsilon , \omega
\right](\mathbf{c},\hat{\mathbf{k}};\hat{\Bx}) \quad \mbox{as }
|\By| \rightarrow \infty.
$$
Since the left-hand sides of the previous equations are coincide, we have \eqnref{farfield3}.

Suppose that $(\mu, \epsilon) $ is a S-vanishing structure of
order $N$ at low frequencies as in Section \ref{mlstructure}. From
\eqnref{farfield1} and \eqnref{farfield3}, we have \beq
\label{far_field_rho_N_1} \bm{A}_{\infty} \left[ \mu \circ
{\Psi}_{\frac{1}{\rho}}, \epsilon \circ \Psi_{\frac{1}{\rho}} ,
\omega \right](\mathbf{c},\hat{\mathbf{k}};\hat{\Bx}) =
o(\rho^{2N+1}) \eeq Then, we define the diffeomorphism $F_{\rho}$
as
\begin{equation*} F_{\rho}(\Bx):= \begin{cases}
\ds \Bx \quad&\mbox{for } |\Bx|\geq2,\\
\ds\Bigr(\frac{3-4\rho}{2(1-\rho)}+\frac{1}{4(1-\rho)}|\Bx|\Bigr)\frac{\Bx}{|\Bx|} \quad&\mbox{for }2\rho\leq|\Bx|\leq 2,\\
\ds\Bigr(\frac{1}{2}+\frac{1}{2\rho}|\Bx|\Bigr)\frac{\Bx}{|\Bx|} \quad&\mbox{for }\rho\leq|\Bx|\leq 2\rho,\\
\ds\frac{\Bx}{\rho} &\mbox{for }|\Bx|\leq \rho.
\end{cases}
\end{equation*}
We then get from \eqnref{far_field_rho_N_1} and Lemma \ref{change_of_var} the main result of this paper.
\begin{theorem}\label{thm:main}
If $(\mu, \epsilon)$ is a S-vanishing structure of order $N$ at
low frequencies, then there exists $\rho_0$ such that
$$
\bm{A}_{\infty} \left[ (F_{\rho})_*( \mu \circ
\Psi_{\frac{1}{\rho}}),(F_{\rho})_* (\epsilon \circ
\Psi_{\frac{1}{\rho}}) , \omega
\right](\mathbf{c},\hat{\mathbf{k}};\hat{\Bx}) = o(\rho^{2N+1}),
$$ for all $\rho \leq  \rho_0$, uniformly in
$(\hat{\mathbf{k}},\hat{\Bx})$.
\end{theorem}
Remark that the cloaking structure $\bigr((F_{\rho})_*( \mu \circ
\Psi_{\frac{1}{\rho}}),(F_{\rho})_* (\epsilon \circ
\Psi_{\frac{1}{\rho}}) \bigr)$ in Theorem \ref{thm:main} satisfies the PEC boundary condition on $|\Bx|=1$.

\section{Conclusion}
We have shown near-cloaking examples for the Maxwell equation. We
have designed a cloaking device that achieves enhanced cloaking
effect based on the method of \cite{AKLL1,AKLL2}  to
electromagnetic scattering problems. Any target placed inside the
cloaking device has an approximately zero scattering amplitude.
Such cloaking device is obtained by the blow up using the
transformation optics of a multi-coated inclusion with PEC
boundary condition. The cloaking device has anisotropic
permittivity and permeability parameters.

\end{document}